\documentclass[final,11pt,3p]{elsarticle}
\usepackage[T1]{fontenc}
\usepackage[english]{babel}
\usepackage{amssymb,amsmath,amsthm}
\usepackage{MnSymbol}
\usepackage{array}   
\newcolumntype{C}{>{$}c<{$}} 
\usepackage{mathtools}
\usepackage{caption}
\usepackage{xcolor}
\usepackage{graphicx}
\usepackage{amssymb}
\usepackage{float}
\usepackage{rotating}
\usepackage{amsthm}
\usepackage{theoremref}
\usepackage{multicol}
\usepackage{multirow}
\usepackage{array}
\usepackage{mathrsfs}
\usepackage[mathcal]{eucal}
\usepackage{makecell}
\usepackage{amstext}
\usepackage[pdf]{graphviz}
\usepackage{longtable}
\newcolumntype{L}{>{$}l<{$}}
\numberwithin{equation}{section}
\theoremstyle{plain}
\newtheorem{theorem}{Theorem}[section]

\newtheorem{proposition}[theorem]{Proposition}
\newtheorem{lemma}[theorem]{Lemma}
\newtheorem{remark}[theorem]{Remark}

\theoremstyle{definition}

\newtheorem{definition}[theorem]{Definition}

\everymath{\displaystyle}
\usepackage{lineno}
\usepackage{changepage}
\newcounter{matcount}
\newcommand{\matitem}[1]{%
  \stepcounter{matcount}%
  (\thematcount)\;
  #1
}
\makeatletter
\def\ps@pprintTitle{%
 \let\@oddhead\@empty
 \let\@evenhead\@empty
 \def\@oddfoot{}%
 \let\@evenfoot\@oddfoot}
\makeatother
\begin{document}
\begin{frontmatter}
\title{Varieties of nilpotent Jordan superalgebras of dimension five} 
\author[1]{Isabel Hern\'andez}
\ead{isabel@cimat.mx}
\author[2]{Laiz Valim da Rocha\corref{cor1}}
\ead{laizvalim@gmail.com}
\author[3]{Rodrigo Lucas Rodrigues}
\ead{rodrigo@mat.ufc.br}
\cortext[cor1]{Corresponding author}
\affiliation[1]{organization={Secretar\'{\i}a de Ciencia,  Humanidades, Tecnolog\'{\i}a e Innovaci\'on and Centro de Investigaci\'on en Matem\'aticas, Unidad M\'erida},
city={M\'erida},
postcode={97302},
country={Mexico}}
\affiliation[2]{organization={Instituto de Matem\'atica e Estat\'{\i}stica, Universidade de S\~ao Paulo, Cidade Universitária},
city={S\~ao Paulo},
postcode={05508-090},
country={Brazil}}
\affiliation[3]{organization={Departamento de Matem\'atica, Universidade Federal do Cear\'a, Campus do Pici, Bloco 914},
city={Fortaleza},
postcode={60440-900},
country={Brazil}}
\begin{abstract}
 The paper is devoted to the description of the varieties of complex 5-dimensional nilpotent Jordan superalgebras. We find all representatives for the isomorphism classes, using the Jordan normal form, results of simultaneous matrix triangularization, the Jordan-Kronecker theorem for a pair of skew-symmetric bilinear forms and similar arguments developed for $\Delta$-modules by Burde and Grunewald. We also provide a complete geometric classification, determining the irreducible components of the corresponding varieties and describing all possible degenerations and non-degenerations between these superalgebras, in particular, applying some $\mathbb{Z}_2$-graded subspaces as invariants. 
\end{abstract}
\end{frontmatter}

\section{Introduction}
The classification of different classes of algebras is a fundamental and very active area of investigation. Even in low dimensions, knowing the list of all non-isomorphic algebras within a given class, through their algebraic classification, is important for understanding their structure, testing conjectures, and producing concrete examples.

For varieties of algebras defined by polynomial identities, another perspective that has gained increasing attention is offered by geometric classification. It frequently involves the study of degenerations or deformations between algebras in the variety under consideration, as well as the identification of rigid ones. The main byproduct of such analysis is the description of their irreducible components. A compilation of results concerning algebraic and geometric classification of several varieties of algebras can be found in \cite{Kay1} and \cite{Kay2}.

Let us consider the variety of Jordan algebras. Introduced in the 1930s \cite{introJor} as part of an effort to formalize the algebraic properties of observables in quantum mechanics, they have been extensively studied since then, establishing connections with distinct areas of mathematics, including Lie theory, projective geometry, and functional analysis (see \cite{Taste}). Since these algebras are described by means of polynomial identities, it is natural to consider the classification problem in the algebraic and geometric perspectives. 

Over an algebraically closed field, the following results are known. In 2005, the variety of Jordan algebras of dimension 3 in characteristic different from 2 was described in \cite{IryShes}. In 2011, the subvariety of nilpotent Jordan algebras of dimensions 3 and 4 was determined in the field of complex numbers in \cite{Anc}. This result was extended in 2014 by \cite{AlgJor4}, where Jordan algebras were classified algebraically, up to dimension four, in characteristic different from 2. Their geometric classification was provided in \cite{Geo4Jor} in the same year. In 2016, the algebraic classification of the nilpotent Jordan algebras of dimension $5$ in characteristic different from 2 was established in \cite{Hegazi}, while the geometric classification was given in 2018 \cite{Geo5Jor}. To our knowledge, this remains the highest dimension in which the subvariety of nilpotent Jordan algebras has been described.

One can also address the classification problem in the context of Jordan superalgebras. In this case, it becomes even more challenging, since several classical results no longer hold (for instance, solvability does not imply nilpotency, even in characteristic different from $2$, see \cite{superalgebrasandcounter}) and an additional difficulty arises from the need to treat separately the different possible dimensions of the graded components. Consequently, studies in this direction are still scarce. However, some recent progress has been achieved for Jordan superalgebras.

In 2017, the algebraic classification of Jordan superalgebras of dimension up to three
over an algebraically closed field of characteristic distinct from 2 was provided in \cite{treeJorsup}. In 2018, degenerations between Jordan superalgebras of dimension $3$ over the field of complex numbers were described in \cite{DegJorsup}. In 2021, a concrete list of non-isomorphic commutative power-associative superalgebras up to dimension 4 over an algebraically closed field of characteristic prime to 30 was provided in \cite{Power}. As a byproduct, the authors obtained the algebraic classification of Jordan superalgebras of dimension $4$. In 2025, the variety of Jordan superalgebras of dimension 4 over an algebraically closed field of characteristic $0$ whose even part has dimension 1 or 3 was described in \cite{4JorsupI}. In \cite{SupJor4}, the authors described all irreducible components of complex  four-dimensional Jordan superalgebras. In \cite{4JorsupII} the variety of Jordan superalgebras of dimension $4$ whose even part has dimension $2$ over an algebraically closed field  of characteristic $0$ was described. Finally, in \cite{completeKay}, the authors presented a short note to answer the open questions posted in \cite{4JorsupI} and \cite{4JorsupII}. Recently, a relevant application of the classification of Jordan superalgebras was presented in \cite{classrightaltsup}, where right alternative superalgebras of dimension 3 are classified using the Jordan superalgebras of the same dimension.

In this paper, we focus on the subvariety of nilpotent Jordan superalgebras and extend their algebraic and geometric classification to dimension 5 over the complex numbers field $\mathbb{C}$. This is the first dimension in which infinite families of non-isomorphic nilpotent Jordan superalgebras appear, which occur precisely when the even part has dimension 2. In the corresponding variety, two of the five irreducible components are given by  the Zariski closure of a union of orbits of an infinite family of superalgebras. 

The paper is organized as follows. Section \ref{prel} contains the necessary preliminaries on Jordan superalgebras, introduces the variety of Jordan superalgebras $\mathcal{JS}^{(m,n)}$ and its nilpotent subvariety $\mathcal{NJS}^{(m,n)}$. In this section, the notion of degeneration between superalgebras in this variety is also established and some useful invariants are indicated to determine when a degeneration does not occur. In particular, the $\mathbb{Z}_{2}$-graded versions of two well-known invariants are presented, namely the graded dimensions of the annihilator and the associative center of a superalgebra. We note that these invariants can also be applied to other classes of superalgebras.

In Section \ref{classalg}, we describe the method used to classify nilpotent Jordan superalgebras, and apply it to obtain a complete algebraic classification. At this stage, it is necessary to adopt distinct strategies for describing the action of the even part on the odd part, depending on their respective dimensions, in order to simplify the computations. A further difficulty that arises when the even part has dimension $2$ is to determine all the non-isomorphic superalgebras in the case that such action is trivial. This problem is overcome due to the Jordan-Kronecker theorem, for which a more elementary proof can be found in \cite{Kronecker}.

Finally, Section \ref{classgeo} is devoted to the geometric classification. We show the technique used to prove degenerations and how to proceed when dealing with a family of superalgebras. Then, we identify all possible degenerations and justify all cases in which degeneration does not occur. We highlight that the invariants presented in Section \ref{prel} are sufficient to identify all rigid superalgebras. Nevertheless, for a complete classification, we follow the strategy used in \cite{seeley} and \cite{SupJor4} to prove non-degeneration of the only remaining case where the invariant criteria cannot be applied. The results are summarized in Table \ref{table:ndegcaso14} to Table \ref{table:degcaso23}. We also include the Hasse diagram of the corresponding variety as a visual tool to illustrate the complete geometric classification. Consequently, the irreducible components are determined.

\section{Preliminaries}\label{prel}
Throughout this paper, we work over the field of complex numbers $\mathbb{C}$. An algebra is called a \emph{Jordan algebra} if it satisfies the identities $xy=yx$ (commutativity) and $(x,y,x^2)=0$ (the Jordan identity), where $(x,y,z)=(xy)z-x(yz)$ denotes the associator of the elements $x,y,z$. Over the field $\mathbb{C}$, the Jordan identity is equivalent to its complete linearization, given by
\[
(wx)(yz) + (wy)(xz) + (wz)(xy) - x(w(yz)) - y(w(xz)) - z(w(xy)) = 0.
\]

A superalgebra $A=A_{0}\oplus A_{1}$ is a $\mathbb{Z}_{2}$-graded algebra, that is, $A_{0}$ and $A_{1}$ are vector subspaces of $A$ (called the even and odd parts of $A$, respectively) such that $A_{i}\cdot A_{j}\subseteq A_{i+j \ \text{mod} \ 2}$. A nonzero element $x \in A_{0}\cup A_{1}$ is said to be homogeneous and we denote by $|x|$ its parity, where $|x|=i$ if $x\in A_{i}$, $i\in\mathbb{Z}_{2}$.

Let $G$ be the Grassmann algebra generated by the elements $1,e_1,\dots,e_{n},\dots$ satisfying $e_{i}^{2}=0$ and $e_{i}e_{j}=-e_{j}e_{i}$. It is easy to see that it has a superalgebra structure with the grading $G=G_{0}\oplus G_{1}$, where $G_{0}$ is spanned by 1 and by the products of an even number of $e_{i}$'s, and $G_{1}$ is spanned by the products of an odd number of $e_{i}$'s. For a superalgebra $A=A_{0}\oplus A_{1}$, we define its \emph{Grassmann envelope} as $G(A)=G_{0}\otimes A_{0} \oplus G_{1}\otimes A_{1} $, where the multiplication is
given by $(x\otimes u)(y\otimes v)=xy\otimes uv,$ 
for every $x \otimes u \in G_{i}\otimes A_{i}, y\otimes v \in G_{j}\otimes A_{j}$, where $i,j \in \{0,1\}$.

Let $\mathcal{V}$ be a variety of algebras defined by a set of multilinear identities (see \cite{Osborn} and \cite{Nearlyasso}). A superalgebra $A=A_{0}\oplus A_{1}$ is called a $\mathcal{V}$-superalgebra if its Grassmann envelope $G(A)$ lies in $\mathcal{V}$. 

\begin{definition}\label{defsup} A \emph{Jordan superalgebra} is a superalgebra $J=J_{0}\oplus J_{1}$ satisfying the following superidentities:
\begin{equation}\label{supcomm}
  xy=(-1)^{|x||y|}yx, \ \text{(supercommutativity)}
\end{equation}
and the Jordan superidentity:
\begin{equation}\label{supidenjordan}
\begin{split}
(wx)(yz)+(-1)^{|x||y|}(wy)(xz)+(-1)^{(|x|+|y|)|z|}(wz)(xy) 
-(-1)^{|w||x|}x(w(yz))\\
-(-1)^{|y|(|w|+|x|)}y(w(xz))-(-1)^{|z|(|w|+|x|+|y|)}z(w(xy))=0,
\end{split}
\end{equation}
for all $x,y,z,w \in (J_{0}\cup J_{1}) \ \backslash \ \{0\}$. 
\end{definition}

The left-hand side of the superidentity \eqref{supidenjordan} will be denoted by $J^{s}(w;x,y,z)$. Using \eqref{supcomm}, it is simple to check that $J^{s}(w;x,y,z)=(-1)^{|x||y|}J^{s}(w;y,x,z)=(-1)^{|x|(|y|+|z|)}J^{s}(w;y,z,x)$.
Given two Jordan superalgebras $J=J_{0}\oplus J_{1}$ and $J^{\prime}=J_{0}^{\prime}\oplus J_{1}^{\prime}$, we say that $J$ and $J^{\prime}$ are \emph{isomorphic} if there exists an isomorphism $\Phi: J \to J^{\prime}$ of algebras that preserves the grading. It is convenient to view $\Phi$ as being a direct sum of its restrictions to the even and odd parts, that is, $\Phi=T\oplus S$, with $T=\Phi|_{J_{0}}$ and $S=\Phi|_{J_{1}}$.

Consider $x\in J$. The linear operator $L_{x}:J\rightarrow J$ given by $L_{x}(y)=xy$ is called the \emph{left multiplication operator by $x$}. Note that if $x\in J_{0}$, then $J_{1}$ is invariant under $L_{x}$.
Throughout the paper, we fix the following convention: if $x\in J_{0}$, $l_{x}$ denotes the restriction of $L_{x}$ to $J_{1}$, that is, $l_{x}=L_{x}|_{J_{1}}$. 

Our work will be organized according to the dimensions of the even and odd parts. Therefore, we distinguish a Jordan superalgebra $J=J_{0}\oplus J_{1}$ as being of type $(m,n)$ if $\mathrm{dim}\,J_{0}=m$ and $\mathrm{dim}\,J_{1}=n$. 

\subsection{The variety $\mathcal{JS}^{(m,n)}$}
Let $V=V_{0}\oplus V_{1}$ be a $\mathbb{Z}_{2}$-graded vector space with a fixed homogeneous basis $\{e_1,\dots,e_{m},f_{1},\dots,f_{n}\}$. We can equip $V$ with a Jordan superalgebra structure by determining a set of structure constants $(c_{ij}^{k},\rho_{ij}^{k},\Gamma_{ij}^{k})\in \mathbb{C}^{m^{3}+2mn^{2}}$, where
\[e_{i}e_{j}=\sum_{k=1}^{m}c_{ij}^{k}e_k, \quad e_{i}f_{j}=\sum_{k=1}^{n}\rho_{ij}^{k}f_{k}, \quad f_{i}f_{j}=\sum_{k=1}^{m}\Gamma_{ij}^{k}e_{k},\]
satisfying the polynomial identities given by supercommutativity and the Jordan superidentity. It follows that the set of all Jordan superalgebras of type $(m,n)$ defines an affine variety in $\mathbb{C}^{m^{3}+2mn^{2}}$, denoted by $\mathcal{JS}^{(m,n)}$. In this way, each point $(c_{ij}^{k},\rho_{ij}^{k},\Gamma_{ij}^{k})\in \mathcal{JS}^{(m,n)}$ corresponds, in the fixed basis, to a Jordan superalgebra of type $(m,n)$ over $\mathbb{C}$.

Since isomorphisms between superalgebras preserve the grading, we consider the natural action of the group $G=\mathrm{GL}_{m}(\mathbb{C})\times\mathrm{GL}_{n}(\mathbb{C})$
on $\mathcal{JS}^{(m,n)}$ by change of basis. Observe that the orbits of this action correspond precisely to the isomorphism classes of Jordan superalgebras of type $(m,n)$. Denote by $O(J)$ the $G$-orbit of $J\in \mathcal{JS}^{(m,n)}$ and by $\overline{O(J)}$ the Zariski closure of $O(J)$. It is well-known that the dimension of $O(J)$ is given by the formula:  
\[\mathrm{dim}\, O(J)= \mathrm{dim}\, G - \mathrm{dim} \,\mathrm{Aut}(J),\]
where $\mathrm{dim}\, G=m^{2}+n^{2}$. 

Let $J,J^{\prime}\in\mathcal{JS}^{(m,n)}$, we say that $J$ \emph{degenerates} to $J^{\prime}$, denoted by $J\rightarrow J^{\prime}$, if $J^{\prime}\in \overline{O(J)}$. In this case, the entire orbit $O(J^{\prime})$ lies in $\overline{O(J)}$.

The concept of degeneration is closely related to the notion of deformation, as considered in \cite{4JorsupI}, in the sense that $J$ degenerates to $J^{\prime}$ corresponds to $J^{\prime}$ deforms into $J$.

Let $J\in \mathcal{JS}^{(m,n)}$. The superalgebra $J$ is called \emph{rigid} if $O(J)$ is a Zariski open subset of $\mathcal{JS}^{(m,n)}$. This notion plays a fundamental role in the study of affine varieties, as the Zariski closure of an open orbit determines an irreducible component of the variety. Moreover, it is well known that any affine variety can be decomposed into a finite number of irreducible components.

\subsection{Invariants}
Given two Jordan superalgebras $J$ and $J^{\prime}$, it may be difficult to determine whether $J$ degenerates to $J^{\prime}$. For this reason, a series of invariants is used to establish conditions under which such a relation may hold. As a consequence, we obtain helpful criteria for the non-existence of degenerations. The invariants listed in the next lemma have already been applied in the study of the varieties of Lie and Jordan superalgebras, see \cite{DegLiesup} and \cite{DegJorsup}. 

First, for a Jordan superalgebra $J=J_{0}\oplus J_{1}$ with structure constants $(c_{ij}^{k},\rho_{ij}^{k},\Gamma_{ij}^{k})$, we denote by $a(J)$ and $F(J)$ the Jordan superalgebras with structure constants $(0,0,\Gamma_{ij}^{k})$ and $(c_{ij}^{k},\rho_{ij}^{k},0)$, respectively. We also define the powers $J^{k}$ inductively by setting
$J^{1}=J, \ J^{k+1}=J^{k}J+J^{k-1}J^{2}+\cdots+JJ^{k}$ for $k\geq 1$. For each $k$, the subspace $J^{k}$ is an ideal of $J$, and we can write $J^{k}=(J^{k})_{0}\oplus (J^{k})_{1}$, where the grading is induced by that of $J$.
\begin{lemma}\label{invariants}
Let $J,J^{\prime}\in\mathcal{JS}^{(m,n)}$. If $J\rightarrow J^{\prime}$, then the following conditions must be satisfied:
\begin{enumerate}[i)]
    \item \label{item:aut} $\mathrm{dim}\,\mathrm{Aut}(J)<\mathrm{dim}\,\mathrm{Aut}(J^{\prime})$.
    \item \label{item:J0} $J_{0}\rightarrow J_{0}^{\prime}$.
    \item $a(J)\rightarrow a(J^{\prime})$.
    \item \label{item:functor} $F(J)\rightarrow F(J^{\prime})$.
    \item \label{item:dimpower} $\mathrm{dim}(J^{k})_{i}\geq \mathrm{dim}({J^{\prime }}^{k})_{i}$, for $i\in \mathbb{Z}_{2}$. 
    \item \label{item:PI} If $J$ is associative, then $J^{\prime}$ is also associative. Moreover, if $J$ satisfies a polynomial identity, then $J^{\prime}$ satisfies the same polynomial identity.
\end{enumerate}
\end{lemma}

In addition, we set $\mathrm{Ann}(J)=\{a\in J \mid aJ=0\}$ and $\mathrm{Z}(J)=\{a\in J \mid (a,J,J)=(J,a,J)=(J,J,a)=0\}$ as the annihilator and the associative center of $J$, respectively. It is straightforward to verify that $\mathrm{Ann}(J)$ and $\mathrm{Z}(J)$ are $\mathbb{Z}_2$-graded subspaces of $J$ (moreover, $\mathrm{Ann}(J)$ is an ideal of $J$). In this way, we write $\mathrm{Ann}(J)=(\mathrm{Ann}(J))_{0}\oplus(\mathrm{Ann}(J))_{1}$ and $\mathrm{Z}(J)=(\mathrm{Z}(J))_{0}\oplus (\mathrm{Z}(J))_{1}$.

This motivates us to derive graded versions of the conditions given in \cite[Fact 4]{Geo4Jor} and \cite[Theorem 9, item (4)]{Geo5Jor}, whose proofs can be directly adapted to the graded case.

\begin{lemma}\label{newinvariants}
Let $J,J^{\prime}\in\mathcal{JS}^{(m,n)}$. If $J\rightarrow J^{\prime}$, then the following conditions must be satisfied:
\begin{enumerate}[i)]
    \item $\mathrm{dim}(\mathrm{Ann}(J))_{i}\leq \mathrm{dim}(\mathrm{Ann}(J^{\prime}))_{i}$, for $i \in \mathbb{Z}_{2}$.
    \item $\mathrm{dim}(\mathrm{Z}(J))_{i}\leq \mathrm{dim}(\mathrm{Z}(J^{\prime}))_{i}$, for $i \in \mathbb{Z}_{2}$.
   \end{enumerate}
\end{lemma}
\begin{remark}\label{nondegalg}
Let $J,J^{\prime}\in \mathcal{JS}^{(m,n)}$ be such that all odd products are zero. Then both have the structure of Jordan algebras. Moreover, if $J\nrightarrow J^{\prime}$ as algebras, then $J\nrightarrow J^{\prime}$ as superalgebras, since degeneration between superalgebras is more restrictive than degeneration between algebras.
\end{remark}
\subsection{The variety $\mathcal{NJS}^{(m,n)}$}

Let $J=J_{0}\oplus J_{1}$ be a Jordan superalgebra. We call $J$ \emph{nilpotent} if there exists a positive integer $k$ such that $J^{k}=0$. The least such number is called the \emph{nilindex} of $J$ and is denoted by $\mathrm{nilindex}(J)$. In particular, note that if $J$ is a nilpotent Jordan superalgebra, then $J_{0}$ is also a nilpotent Jordan algebra. 

Since the nilpotency condition can be expressed by polynomial equations in the structure constants, the set of all nilpotent Jordan superalgebras of type $(m,n)$ is an affine subvariety of $\mathcal{JS}^{(m,n)}$, denoted by $\mathcal{NJS}^{(m,n)}$. Consequently, all the notions and results presented in the previous section can be applied independently to this variety. Also, from item \ref{item:dimpower}) of Lemma \ref{invariants}, we have that if $J$ degenerates to $J^{\prime}$, then $\mathrm{nilindex}(J)\geq \mathrm{nilindex}(J^{\prime})$.

Our aim is to describe the varieties $\mathcal{NJS}^{(m,n)}$, with $m+n=5$. We start by establishing all nilpotent Jordan superalgebras of each type, up to isomorphism.

\section{The algebraic classification}\label{classalg}
\subsection{Types (5,0), (0,5) and (4,1)}
If $J=J_{0}\oplus J_{1}$ is a nilpotent Jordan superalgebra of type $(5,0)$, then it corresponds to a nilpotent Jordan algebra of dimension $5$, which has been algebraically classified in \cite{Hegazi}. On the other hand, if $J$ is of type $(0,5)$, then it has trivial multiplication, and $J$ is the unique superalgebra of this type.

Now, let $J=J_{0}\oplus J_{1}$ be a nilpotent Jordan superalgebra over $\mathbb{C}$ of type $(4,1)$, with $\{e_1,e_2,e_3,e_4$\} and $\{f_{1}\}$ as bases of $J_{0}$ and $J_{1}$, respectively. From the nilpotency of $J$, we see that each $L_{e_{i}}$ is a nilpotent operator for all $i\in\{1,\dots,4\}$. Since $\mathrm{dim}\, J_{1}=1$ and from the nilpotency of $L_{e_{i}}$, we have $e_{i}f_{1}=0$. Moreover, $f_{1}^{2}=0$. Thus, $J_{0}J_{1}=J_{1}J_{1}=0$, and the multiplication of $J$ is completely determined by that of $J_{0}$. Nilpotent Jordan algebras of dimension up to $4$ over the complex field were classified in \cite{Anc}. The same isomorphism classes were obtained in \cite{Hegazi} over an arbitrary algebraically closed field of characteristic different from $2$ and $3$. Throughout this paper, we adopt the notation of \cite{Hegazi}. Therefore, we have, up to isomorphism, the following nilpotent Jordan superalgebras:

{\footnotesize
\begin{longtable}{|l|l|l|l|l|l|}
\caption{Nilpotent Jordan superalgebras of type $(4,1)$ }\label{table:caso41} \\\hline
$J$&\text{multiplication table}&\text{Obs} & \text{nilindex} & $\mathrm{Aut}(J)$&  $\mathrm{Ann}(J)$ \\\hline
$(4,1)_1$  &$e_{i}e_{j}=0$ for $i,j\in\{1,\dots,4\}$. & \text{A}&$2$&$17$&$(4,1)$  \\ \hline 
$(4,1)_2$&$e_1^2=e_2$. & \text{A}&$3$&$11$&$(3,1)$  \\ \hline 
$(4,1)_3$& $e_1e_2=e_3$. &\text{A}&$3$&$9$&$(2,1)$\\\hline
$(4,1)_4$& $e_1^2=e_2, \ e_1e_2=e_3$.&\text{A}&$4$&$7$&$(2,1)$\\ \hline
$(4,1)_5$& $e_1e_2=e_4, \ e_3^2=e_4$. & \text{A}&$3$&$8$&$(1,1)$ \\ \hline
$(4,1)_6$&  $e_1^2=e_2, \ e_2e_3=e_4$. & \text{NA}&$4$&$5$&$(1,1)$\\ \hline
$(4,1)_7$& $e_1^2=e_2, \ e_1e_2=e_4, \ e_3^2=e_4$.& \text{A}&$4$&$6$&$(1,1)$\\ \hline
$(4,1)_8$& $e_1e_2=e_3, \ e_1e_3=e_4$. & \text{NA}& $4$&$6$&$(1,1)$\\ \hline
$(4,1)_9$& $e_1e_2=e_3, \ e_1e_3=e_4, \ e_2^2=e_4$.& \text{NA}& $4$&$5$&$(1,1)$\\ \hline
$(4,1)_{10}$& $e_1e_2=e_3, \ e_1e_3=e_4, \ e_2e_3=e_4$. & \text{NA}&$4$&$4$&$(1,1)$\\ \hline
$(4,1)_{11}$& $ e_1^2=e_2, \ e_1e_2=e_3, \ e_1e_3=e_4, \ e_2^2=e_4 $.&\text{A}&$5$&$5$&$(1,1)$\\ \hline
$(4,1)_{12}$& $e_1^2=e_3, \ e_1e_2=e_4$. & \text{A}&$3$&$8$&$(2,1)$ \\ \hline
$(4,1)_{13}$& $e_1^2=e_3, \ e_2^2=e_4 $.& \text{A} &$3$&$7$&$(2,1)$\\ \hline
\end{longtable}
}

Table \ref{table:caso41} presents additional information. In Column $3$ we write ``A'' if $J$ is an associative superalgebra and ``NA'' otherwise. Column $4$ indicates the nilindex of $J$, while Columns $5$ and $6$ show, respectively, the dimension of the automorphism group of $J$, denoted by $\mathrm{Aut}(J)$, and the type of the annihilator of $J$. Henceforth, we maintain this notation.

 For the remaining types, we adopt the following strategy: 
\begin{enumerate}[i)]
    \item The even part $J_{0}$ is fixed according to the classification of nilpotent Jordan algebras of dimension less than 4 obtained in \cite{Hegazi}  and presented in Table \ref{table:algjor4}. \label{step1}
    
\item We describe the nilpotent action of $J_{0}$ on $J_{1}$ in terms of a suitable basis, using Jordan's normal form or simultaneous triangularization. \label{step2}

\item We determine all the possible structures of nilpotent Jordan superalgebras in $J=J_{0}\oplus J_{1}$. \label{step3}

\item We find representatives for the isomorphism classes of such nilpotent Jordan superalgebras.
\end{enumerate}
\begin{small}

\begin{longtable}{|l|l|l|}
 \caption{Nilpotent Jordan algebras of dimension $\leq 3.$} \label{table:algjor4}\\ \hline
 $J_{0}$&multiplication table & dim $J_{0}$ \\ \hline
$J_{1,1}$&$e_{1}^{2}=0$.&$1$  \\ \hline
$J_{2,1}$&$e_{i}e_{j}=0, \ i,j\in\{1,2\}.$& $2$ \\ \hline
$J_{2,2}$&$e_{1}^{2}=e_{2}.$&2 \\ \hline
$J_{3,1}$&$e_{i}e_{j}=0, \ i,j\in\{1,2,3\}.$&$3$ \\ \hline
$J_{3,2}$&$   e_1^2=e_2$.& $3$ \\ \hline
$J_{3,3}$&$e_{1}e_{2}=e_{3}$.&3 \\ \hline
$J_{3,4}$&$e_{1}^{2}=e_{2}, \ e_{1}e_{2}=e_{3}$.&3 \\ \hline
\end{longtable}
\end{small}
We say that superalgebras obtained in step (\ref{step3}) are based on $J_{0}$. In order to guarantee that we obtain only Jordan superalgebras that are nilpotent in step (\ref{step3}), we use the fact that $L_{f}$, $f\in J_{1}$, is a nilpotent operator on $J$. Since the nilpotency of $L_{e}$, $e\in J_{0}$, is also assumed in steps (\ref{step1}) and (\ref{step2}), it follows from \cite[Theorem 2.5]{EngelJordan} that $J$ is nilpotent.

To determine representatives of the isomorphism classes of superalgebras based on $J_{0}$, we establish the following lemma. First, note that we can represent the product of odd elements by the skew-symmetric bilinear operator $\Gamma:J_{1}\times J_{1}\rightarrow J_{0}$ given by $\Gamma(f_{i},f_{j})=\textstyle\sum_{k=1}^{m}\Gamma_{ij}^{k}e_{k}$. Moreover, $\Gamma$ can be identified with $(\Gamma^{1},\dots,\Gamma^{m})$, where $\Gamma^{k}:J_{1}\times J_{1}\rightarrow \mathbb{C}$ is given by $\Gamma^{k}(f_{i},f_{j})=\Gamma_{ij}^{k}$. 

\begin{lemma}\label{iso14}
 Let $J$ and $J^{\prime}$ be Jordan superalgebras of type $(m,n)$ based on $ J_0$, written with respect to the homogeneous basis $\{e_{1},\dots,e_{m},\allowbreak f_{1},\dots,f_{n}\}$. Then $J\cong J^{\prime}$ if and only if there exist $T\in \mathrm{Aut}(J_{0})$ and an invertible linear map  $S:J_1 \to J_1^{\prime}$ such that:
 \begin{align}
  \sum_{i=1}^{m}T_{ik}l_{e_{i}}^{\prime}&=Sl_{e_{k}}S^{-1}, \label{isosimilar}\\  
  \Gamma^{\prime k}&=\sum_{i=1}^{m}T_{ki}(S^{-1})^{t}\Gamma^{i}S^{-1}, \ k\in \{1,\dots,m\}.\label{isocongruent}
 \end{align}
where $T(e_k)= \sum _{i=1}^m T_{ik}e_i$, $l_{e_{i}}^{\prime}:J_{1}^{\prime}\rightarrow J_{1}^{\prime}$ denotes the left multiplication by $e_{i}$ on $J_{1}^{\prime}$, and $\Gamma^{\prime }:J_{1}^{\prime}\times J_{1}^{\prime}\rightarrow J_{0}^{\prime}$ is the bilinear map defined above.
\end{lemma}

The maps $T$ and $S$ in Lemma \ref{iso14} induce a change of basis under which the multiplication of $J$ with respect to the new basis corresponds to the multiplication of $J^{\prime}$ with respect to the basis $\{e_{1},\dots,e_{m},f_{1},\dots,f_{n}\}$. For simplicity, in what follows, we only exhibit the change of basis required to verify the isomorphisms, omitting the vectors that remain fixed.


\subsection{Type (1,4)}

\begin{theorem}
 Let $J=J_{0}\oplus J_{1}$ be a nilpotent Jordan superalgebra of type $(1,4)$. Then $J$ is isomorphic to one of the pairwise non-isomorphic superalgebras described in Table \ref{table:caso14}.  
\end{theorem}
{\footnotesize
\begin{longtable}{|l|l|l|l|l|l|}
 \caption{Nilpotent Jordan superalgebras of type $(1,4)$}\label{table:caso14}\\\hline
 $J$&multiplication table & Obs & nilindex & $\mathrm{Aut}(J)$& $\mathrm{Ann}(J)$ \\\hline
 $(1,4)_1$& $e_{1}^{2}=e_{1}f_{i}= f_{i}f_{j}=0, \ i,j\in\{1,\dots,4\}$. & A & $2$ & $17$ & $(1,4)$ \\ \hline
 $(1,4)_2$&$f_1f_2=e_{1}$. & A&$3$ & $12$ &$(1,2)$ \\ \hline 
 $(1,4)_3$&$f_1f_2=e_{1}, \ f_3f_4=e_{1}$. & A&$3$&$11$&$(1,0)$ \\ \hline
 $(1,4)_4$& $e_{1}f_2=f_1$. & A &$3$ &$11$ &$(0,3)$ \\ \hline
 $(1,4)_{5}$& $e_{1}f_2=f_1, \ f_2f_3=e_{1}$.&NA&$4$&$9$&$(0,2)$\\\hline
 $(1,4)_{6}$&$e_{1}f_2=f_1, \ f_3f_4=e_{1}$.&NA&$4$ & $8$&$(0,1)$\\ \hline
  $(1,4)_{7}$&  $e_{1}f_2=f_1, \ e_{1}f_3=f_2$.& NA & $4$ & $7$&$(0,2)$ \\ \hline
 $(1,4)_{8}$&$e_{1}f_2=f_1, \ e_{1}f_4=f_3$.& A&$3$&$9$&$(0,2)$   \\ \hline 
 $(1,4)_{9}$&$e_{1}f_2=f_1, \ e_{1}f_4=f_3, \ f_2f_4=e_{1}$.  & NA&$4$&$8$ &$(0,2)$ \\ \hline

\end{longtable}
}
\begin{proof}
Since $\mathrm{dim}\,J_{0}=1$, there is just one possibility for $J_{0}$, namely $J_{0}=J_{1,1}$. Furthermore, since $l_{e_1}$ is nilpotent, there exists a basis $\mathcal{B}_{1}=\{f_{1},f_{2},f_{3},f_{4}\}$ of $J_{1}$ such that $[l_{e_1}]_{\mathcal{B}_{1}}$ can be written in Jordan normal form:
\[[l_{e_1}]_{\mathcal{B}_{1}}=\left(\begin{array}{cccc}
0 & \delta_{1} & 0 & 0\\
0 & 0 & \delta_{2} & 0 \\ 
0 & 0 & 0 & \delta_{3} \\
0 & 0 & 0 & 0\\
\end{array}\right), \ \delta_{1},\delta_{2},\delta_{3}\in\{0,1\}\]

 It remains to determine the products $f_{i}f_{j}=\omega_{ij}e_{1}$, $\omega_{ij}\in\mathbb{C}, \ 1\leq i<j\leq4$. Up to matrix similarity, the triple $\delta=(\delta_{1},\delta_{2},\delta_{3})$ can assume one of the following values:
\begin{enumerate}[1)]
    \item $\delta=(0,0,0)$:
    
In this case, the Jordan superidentity and the nilpotency of the $L_f$, $f \in J_1$, impose no restrictions on the parameters $\omega_{ij}$. Therefore, the map $\omega:J_1\times J_1\to \mathbb{C}$, given by $\omega(f_i,f_j)=\omega_{ij}$, is an arbitrary skew-symmetric bilinear form. By Lemma \ref{iso14}, it follows that the problem of finding representatives is equivalent to determining all possible skew-symmetric bilinear forms up to congruence. The complete description of such forms follows as a consequence of \cite[Chapter XV, Theorem 8.1]{Lang}. Hence, with a suitable change of basis in $J_1$, $J$ is isomorphic to one of the following: $(1,4)_1$, $(1,4)_2$ or $(1,4)_3$.    
\item $\delta=(1,0,0)$:
     
The nilpotency of $L_{f_{2}},\, L_{f_{2}+f_{3}}$ and $L_{f_{2}+f_{4}}$ gives us $\omega_{12}=\omega_{13}=\omega_{14}=0$, respectively, while the Jordan superidentity does not produce more conditions on $\omega_{ij}$. Hence, if $\omega_{23}=\omega_{24}=\omega_{34}=0$, we have $J\cong (1,4)_4$. Otherwise, there exists $\omega_{i_{0}j_{0}}\in \{\omega_{23},\omega_{24},\omega_{34}\}$ such that $\omega_{i_{0}j_{0}}\neq 0$. In this case, we analyze the value of $\omega_{34}$ to obtain non-isomorphic representatives. 
    \begin{enumerate}[a)]

\item  If $\omega_{34}=0$, we can take a new basis in $J_{1}$ obtained by the change of basis  $e_{1}^{\prime}\mapsto\omega_{24}e_{1}, \ f_{1}^{\prime}\mapsto \omega_{24}f_1,  \ f_{3}^{\prime}\mapsto f_{4}$, $ f_{4}^{\prime}\mapsto f_{3}-\omega_{23}\omega_{24}^{-1}f_{4}$, if $\omega_{24}\neq0$, or $e_{1}^{\prime}\mapsto \omega_{23}e_{1}$, $ f_{1}^{\prime}\mapsto \omega_{23}f_{1}$, if $\omega_{24}=0$, and we get $J\cong (1,4)_{5}$ in both cases.
        \item   If $\omega_{34}\neq 0$, making a change of basis given by $e_{1}^{\prime}\mapsto -\omega_{34}e_{1}, \ f_{1}^{\prime}\mapsto-\omega_{34}f_1, \ f_{2}^{\prime}\mapsto f_2-\omega_{24}\omega_{34}^{-1}f_3+\omega_{23}\omega_{34}^{-1}f_4, \ f_3^{\prime}\mapsto f_4$, $f_4^{\prime}\mapsto f_3$,  we have $J\cong (1,4)_{6}$.
    \end{enumerate}
    
\item $\delta=(1,1,0)$: 

The Jordan superidentity implies that $\omega_{ij}=0$ for every $i,j$. In fact, $J^{s}(e_{1};e_{1},f_2,f_3)=3\omega_{12}e_{1}$, $J^{s}(e_{1};e_{1},f_3,f_4)=-\omega_{14}e_{1}$, $J^{s}(e_{1};f_2,f_3,f_4)=-\omega_{24}f_1$, $J^{s}(f_3;e_{1},f_1,f_3)=-\omega_{13}f_1$, $J^{s}(f_3;e_{1},f_2,f_3)=-2\omega_{23}f_1+\omega_{13}f_{2}$ and $J^{s}(f_3;e_{1},f_3,f_4)=-\omega_{34}f_1$. Consequently, $J\cong (1,4)_{7}$.
  
\item $\delta=(1,0,1)$: 

    From the Jordan superidentity it follows that $\omega_{ij}=0$ for all $(i,j)$ such that $(i,j)\neq (2,4)$. In fact, $J^{s}(e_{1};e_{1},f_2,f_4)=2\omega_{13}e_{1}$, 
 $J^{s}(f_2;e_{1},f_2,f_4)=-\omega_{23}f_1+\omega_{12}f_{3}$ and $J^{s}(f_4;e_{1},f_2,f_4)=-\omega_{34}f_1+\omega_{14}f_{3}$. No further restrictions are required by the Jordan superidentity and the nilpotency of $L_f$ for any $f\in J_1$. Therefore, if $\omega_{24}=0$, $J\cong (1,4)_{8}$. Otherwise, after the change of basis defined by $e_{1}^{\prime}\mapsto \omega_{24}e_{1}$, $f_1^{\prime}\mapsto\omega_{24}f_1$, $f_3^{\prime}\mapsto \omega_{24} f_3$, we obtain $J\cong (1,4)_{9}$.
     
\item $\delta=(1,1,1)$: 

        Observe that $J^{s}(e_{1};e_{1},e_{1},f_4)=-2\delta_1\delta_2\delta_3f_1$, so we must have $\delta_1\delta_2\delta_3=0$, which implies that $\delta$ cannot take this value.
\end{enumerate}
\end{proof}

In what follows, we only display substitutions in $J^{s}(w;x,y,z)$ that yield non-trivial conditions for $J$ to admit a Jordan superalgebra structure. All remaining substitutions either vanish identically or lead to conditions already obtained.
\subsection{Type (3,2):}
\begin{theorem}
 Let $J=J_{0}\oplus J_{1}$ be a nilpotent Jordan superalgebra of type $(3,2)$. Then $J$ is isomorphic to one of the pairwise non-isomorphic superalgebras described in Table \ref{table:caso32}.  
\end{theorem}
\begin{small}
\begin{longtable}{|l|l|l|l|l|l|}
\caption{Nilpotent Jordan superalgebras of type $(3,2)$}\label{table:caso32}\\ \hline
$J$&\text{multiplication table}&\text{Obs} & \text{nilindex} &$\mathrm{Aut}(J)$& $\mathrm{Ann}(J) $ \\\hline
  $(3,2)_{1}$ & $ e_{i}e_{j}=e_{i}f_{k}= f_{k}f_{l}=0, \ i,j\in \{1,2,3\}, \ k,l \in\{1,2\}.$ & \text{A}&$2$&$13$&$(3,2)$ \\ \hline
  $ (3,2)_{2}$ & $f_1f_2=e_2 $.&\text{A}&$3$&$10$&$(3,0)$ \\ \hline
  $(3,2)_{3}$ & $e_1f_2=f_1$. &\text{A}&$3$&$9$&$(2,1)$ \\ \hline
  $(3,2)_{4}$ & $ e_1f_2=f_1, \ f_1f_2=e_2$. &\text{NA}&$4$&$7$&$(2,0)$ \\ \hline \hline

 $ (3,2)_{5} $&$e_1^2=e_2 $.&\text{A}&$3$&$9$&$(2,2)$ \\ \hline
$  (3,2)_{6}$&$ e_1^2=e_2, \ f_1f_2=e_1$. &\text{NA}&$5$&$6$&$(2,0)$ \\\hline
  $ (3,2)_{7}$ &$ e_1^2=e_2, \ f_1f_2=e_2$. & \text{A}&$3$&$8$&$(2,0)$\\ \hline
 $ (3,2)_{8}$ & $e_1^2=e_2, \ f_1f_2=e_3$. & \text{A}&$3$&$7$&$(2,0)$\\ \hline

  $  (3,2)_{9}$& $e_1^2=e_2, \ e_2f_2=f_1 $.& \text{NA}&$4$&$5$&$(1,1)$\\ \hline
 $ (3,2)_{10}$& $e_1^2=e_2, \ e_3f_2=f_1$. & \text{A}&$3$&$6$&$(1,1)$\\ \hline
 $ (3,2)_{11}$& $e_1^2=e_2, \ e_3f_2=f_1, \ f_1f_2=e_2 $.&\text{NA}&$4$&$5$&$(1,0)$ \\ \hline

 $ (3,2)_{12}$&$ e_1^2=e_2, \ e_1f_2=f_1$. & \text{A}&$3$&$7$&$(2,1)$ \\ \hline
$  (3,2)_{13}$&$ e_1^2=e_2, \ e_1f_2=f_1, \ f_1f_2=e_2$.&\text{NA}&$4$&$6$&$(2,0)$\\ \hline
 $  (3,2)_{14}$& $e_1^2=e_2, \ e_1f_2=f_1, f_1f_2=e_3 $.&\text{NA}&$4$&$5$&$(2,0)$ \\ \hline\hline

 $ (3,2)_{15}$ & $e_1e_2=e_3$.&\text{A}&$3$&$8$&$(1,2)$ \\\hline
$  (3,2)_{16}$& $e_1e_2=e_3, \ f_1f_2=e_1$. & \text{NA}&$4$&$6$&$(1,0)$\\ \hline
   $(3,2)_{17}$& $e_1e_2=e_3, \ f_1f_2=e_3 $.&\text{A}&$3$&$7$&$(1,0)$ \\ \hline
$  (3,2)_{18}$& $e_1e_2=e_3, \ f_1f_2=e_1+e_2 $.& \text{NA}&$5$&$5$&$(1,0)$\\  \hline

$  (3,2)_{19}$&$ e_1e_2=e_3, \ e_3f_2=f_1$. &\text{NA}&$4$&$4$&$(0,1)$ \\   \hline
$  (3,2)_{20}$& $e_1e_2=e_3, \ e_1f_2=f_1$.&\text{A}&$3$&$6$&$(1,1)$ \\ \hline
 $ (3,2)_{21}$ & $e_1e_2=e_3, \ e_1f_2=f_1, \ f_1f_2=e_3 $.& \text{NA}&$4$&$5$&$(1,0)$\\ \hline
   $(3,2)_{22}$ &  $e_1e_2=e_3, \ e_1f_2=e_2f_2=f_1 $.& \text{A}&$3$&$5$&$(1,1)$\\  \hline
 $  (3,2)_{23}$& $e_1e_2=e_3, \ e_1f_2=e_2f_2=f_1, \ f_1f_2=e_3$.&\text{NA}&$4$&$4$&$(1,0)$ \\ \hline\hline

 $ (3,2)_{24}$& $e_1^2=e_2, \ e_1e_2=e_3$.&\text{A}&$4$&$7$&$(1,2)$\\ \hline
  $(3,2)_{25}$&$ e_1^2=e_2, \ e_1e_2=e_3, \ f_1f_2=e_2$.&\text{NA}&$4$&$5$&$(1,0)$\\ \hline
 $ (3,2)_{26}$&$e_1^2=e_2, \ e_1e_2=e_3, \ f_1f_2=e_3 $.&\text{A}&$4$&$6$&$(1,0)$ \\ \hline

  $(3,2)_{27} $& $e_1^2=e_2, \ e_1e_2=e_3, \ e_2f_2=f_1$.&\text{NA} &$4$&$4$&$(1,1)$\\ \hline
 $ (3,2)_{28}$&$ e_1^2=e_2, \ e_1e_2=e_3, \ e_1f_2=f_1$. &\text{A}&$4$&$5$&$(1,1)$\\ \hline
  $(3,2)_{29}$& $e_1^2=e_2, \ e_1e_2=e_3, \ e_1f_2=f_1, \ f_1f_2=e_3 $.&\text{NA}&$4$&$4$&$(1,0)$\\ \hline
\end{longtable}
\end{small}

\begin{proof} 
Indeed, $J_{0}$ is a nilpotent Jordan algebra of dimension $3$ and we fix a basis $\{e_1,e_2,e_3\}$. Since $J$ is nilpotent, each linear operator in the algebra generated by $l_{e_1},l_{e_2},l_{e_3}$ is nilpotent. It follows from \cite[Theorem 1.3.1]{radjavi} that there exists a basis $\mathcal{B}_{1}=\{f_1,f_2\}$ of $J_{1}$ such that $l_{e_1},l_{e_2},l_{e_3}$ are simultaneously triangularizable. That is:
\[[l_{e_1}]_{\mathcal{B}_{1}}=\left(\begin{array}{cc}
0 & \delta_{1} \\
0 & 0  \\ 
\end{array}\right), \ [l_{e_2}]_{\mathcal{B}_{1}}=\left(\begin{array}{cc}
0 & \delta_{2} \\
0 & 0  \\ 
\end{array}\right), \ [l_{e_3}]_{\mathcal{B}_{1}}=\left(\begin{array}{cc}
0 & \delta_{3} \\
0 & 0  \\ 
\end{array}\right),\]
so we need to determine:
\begin{equation}\label{parameters32}
f_1f_2=\omega_{1} e_1+\omega_{2} e_2 +\omega_{3} e_3 \ \text{and } e_if_2=\delta_if_1, \ i=1,\dots 3. 
\end{equation}

We denote $\delta=(\delta_{1},\delta_{2},\delta_{3})$ and $\omega=(\omega_{1},\omega_{2},\omega_{3})$. Note that the nilpotency of $L_{f_{2}}$ implies that the trace of $[L_{f_{2}}^{2}]_\mathcal{B}$ must be zero, where $\mathcal{B}=\{e_1,e_2,e_3,f_1,f_2\}$.
This yields the following condition, independently of the choice of $J_{0}$:
\begin{equation}\label{eigenj31}
\delta_1\omega_1+\delta_2\omega_2+\delta_3\omega_3=0  . 
\end{equation}
Moreover, by Table \ref{table:algjor4} we have four non-isomorphic possibilities for $J_{0}$. We will analyze each case.

\underline{$J_{0}=J_{3,1}$:} The Jordan superidentity does not produce more restrictions. Considering the condition given in \eqref{eigenj31}, we divide into the following cases: 
 \begin{enumerate}[1)]
 \item $\delta=0$.
 \begin{enumerate}[a)]
     \item If $\omega=0$, then $J$ is isomorphic to $(3,2)_{1}$.
     \item If $\omega\neq 0$, then $J$ is isomorphic to $(3,2)_{2}$ via the change of basis:
     \begin{enumerate}[i)]
         \item $e_{2}^{\prime}\mapsto \omega_{1}e_{1}+\omega_{2}e_{2}+\omega_{3}e_{3}$, if $\omega_{2}\neq 0$.
         \item $e_{1}^{\prime}\mapsto e_{2}, \ e_{2}^{\prime}\mapsto \omega_{1}e_{1}+\omega_{3}e_{3}$, if $\omega_{2}=0$ and $\omega_{1}\neq0$.
         \item $e_{2}^{\prime}\mapsto\omega_{3}e_{3}, \ e_{3}^{\prime}\mapsto e_{2}$, if $\omega_{2}=\omega_{1}=0$ (consequently, $\omega_{3}\neq0$).
     \end{enumerate}
     \end{enumerate}
      \item $\delta\neq 0$. In this case, with a suitable change of basis, we can verify that $J$ is isomorphic to a superalgebra $J^{\prime}$ such that $\delta^{\prime}=(1,0,0)$ and $\omega^{\prime}=(\omega^{\prime}_{1},\omega^{\prime}_{2},\omega^{\prime}_{3})$, with $\omega_{i}^{\prime}\in\mathbb{C}$, for all $i=1,2,3$. Indeed:
      \begin{enumerate}[a)]
          \item If $\delta_{1}\neq0$, by changing $e_{1}^{\prime}\mapsto\delta_{1}^{-1}e_{1}, \ e_{2}^{\prime}\mapsto e_{2}-\delta_{2}\delta_{1}^{-1}e_{1}, \ e_{3}^{\prime}\mapsto e_{3}-\delta_{3}\delta_{1}^{-1}e_{1}$.
          \item If $\delta_{1}=0$ and $\delta_{2}\neq 0$, by changing $e_{1}^{\prime}\mapsto \delta_{2}^{-1}e_{2}, \ e_{2}^{\prime}\mapsto e_{1}-\delta_{1}\delta_{2}^{-1}e_{2}, \ e_{3}^{\prime}\mapsto e_{3}-\delta_{3}\delta_{2}^{-1}e_{2}$.
          \item If $\delta_{1}=\delta_{2}=0$, we may assume $\delta_{3}\neq 0$. In this case, the isomorphism is obtained via $e_{1}^{\prime}\mapsto \delta_{3}^{-1}e_{3}, \ e_{3}^{\prime}\mapsto e_{1}$.
      \end{enumerate}
 \end{enumerate}
  Thus, it remains to consider the following possibilities. If $\omega^{\prime}=0$, then $J$ is isomorphic to $(3,2)_{3}$. If $\omega^{\prime}\neq 0$, then from \eqref{eigenj31} we get $\omega_{1}^{\prime}=0$, while $\omega_{2}^{\prime}$ and $\omega_{3}^{\prime}$ remain free. By replacing $e_{2}^{\prime\prime}\mapsto \omega_{2}^{\prime}e_{2}^{\prime}+\omega_{3}^{\prime}e_{3}^{\prime}$, we have $J\cong(3,2)_{4}$.

\underline{$J_{0}=J_{3,2}$:} From $J^{s}(f_2;e_1,e_1,f_2)=J^{s}(f_2;e_1,e_3,f_2)=0$ together with the condition \eqref{eigenj31} we get the following system:
\begin{equation} \label{sysJ32}
\begin{alignedat}{2}
\delta_{2}\begin{pmatrix}\omega_{1}\\ \omega_{3}\end{pmatrix} &= 0, &\qquad \delta_{3}\,\omega_{1} &= 0, \\
2\delta_{1}\,\omega_{1}-\delta_{2}\,\omega_{2} &= 0, &\qquad \delta_{1}\,\omega_{1}+\delta_{2}\,\omega_{2}+\delta_{3}\,\omega_{3} &= 0. 
\end{alignedat}
\end{equation}

We consider the following possibilities:

\begin{enumerate}[1)]
 \item $\delta=0$. Hence $\omega_{i}$ are arbitrary for all $i\in\{1,2,3\}$. In this case, we have four possibilities:
  \begin{enumerate}[a)]
     \item If $\omega=0$, then $J$ is isomorphic to $(3,2)_{5}$.
     \item If $\omega_{1}\neq 0$, by changing $e_{1}^{\prime}\mapsto \omega_{1}e_{1}+\omega_{2}e_{2}+\omega_{3}e_{3}, \ e_{2}^{\prime}\mapsto \omega_{1}^{2}e_{2}$, we obtain $J\cong (3,2)_{6}$.
      \item If $\omega_{1}=\omega_{3}=0$, we may assume $\omega_{2}\neq 0$. By replacing $f_{2}^{\prime}\mapsto \omega_{2}^{-1}f_2$, we have $J\cong (3,2)_{7}$.
     \item If $\omega_{1}=0$ and $\omega_{3}\neq 0$, by replacing $e_{3}^{\prime}\mapsto \omega_{2}e_{2}+\omega_{3}e_{3}$, we have $J\cong (3,2)_{8}$.
    
 \end{enumerate}

\item $\delta\neq0$.
\begin{enumerate}[a)]
    \item Assume that $\delta_{2}\neq0$. Then from \eqref{sysJ32}, $\omega=0$. In this case, we have that $J$ is isomorphic to $(3,2)_{9}$ via the change of basis: 
    \begin{enumerate}[i)]
            \item $e_{1}^{\prime}\mapsto e_{1}-\delta_{1}\delta_{3}^{-1}e_{3}, \ e_{3}^{\prime}\mapsto \delta_{3}e_{2}-\delta_{2}e_{3}, \ f_{1}^{\prime}\mapsto \delta_{2}f_{1}$, if $\delta_{3}\neq 0$.
            \item $e_{1}^{\prime}\mapsto e_{1}-\delta_{1}\delta_{2}^{-1}e_{2}, \ f_{1}^{\prime}\mapsto \delta_{2}f_{1}$, if $\delta_{3}=0$.
        \end{enumerate}

    \item Assume $\delta_{2}=0$ and $\delta_{3}\neq 0$. Then from \eqref{sysJ32}, we get $\omega_{1}=\omega_{3}=0$, while $\delta_{1}$ and $\omega_{2}$ are arbitrary.
        \begin{enumerate}[i)]
            \item If $\omega_{2}=0$, by changing $e_{1}^{\prime}\mapsto e_{1}-\delta_{1}\delta_{3}^{-1}e_{3}, \ f_{1}^{\prime}\mapsto \delta_{3}f_{1}$, we obtain $J\cong(3,2)_{10}$.
            \item If $\omega_{2}\neq 0$,  by changing $e_{1}^{\prime}\mapsto e_{1}-\delta_{1}\delta_{3}^{-1}e_{3}, \ e_{3}^{\prime}\mapsto (\omega_{2}\delta_{3})^{-1}e_{3}, \ f_{1}^{\prime}\mapsto \omega_{2}^{-1}f_{1}$, we obtain $J\cong (3,2)_{11}$.
            \end{enumerate}

    \item If $\delta_{2}=\delta_{3}=0$ and $\delta_{1}\neq 0$. From \eqref{sysJ32}, we get $\omega_{1}=0$, whereas $\omega_{2}$ and $\omega_{3}$ remain free.
        \begin{enumerate}[i)]
            \item If $\omega_{2}=\omega_{3}=0$, by replacing $f_{1}^{\prime}\mapsto \delta_{1}f_{1}$, we have $J\cong(3,2)_{12}$.
            \item If $\omega_{2}\neq0$ and $\omega_{3}=0$, by changing $f_{1}^{\prime}\mapsto \sqrt{\delta_{1}\omega_{2}^{-1}}f_{1}, \ f_{2}^{\prime}\mapsto (\sqrt{\delta_{1}\omega_{2}})^{-1}f_{2}$, we obtain $J\cong (3,2)_{13}$.
            \item If $\omega_{3}\neq 0$, by changing $e_{1}^{\prime}\mapsto \delta_{1}^{-1}e_{1}, \ e_{2}^{\prime}\mapsto \delta_{1}^{-2}e_{2}, \ e_{3}^{\prime}\mapsto \omega_{2}e_{2}+\omega_{3}e_{3}$, we obtain $J\cong (3,2)_{14}$.
        \end{enumerate}
\end{enumerate}

 \end{enumerate}

\underline{$J_{0}=J_{3,3}$:} From $J^{s}(f_2;e_1,e_1,f_2)=J^{s}(f_2;e_1,e_2,f_2)=J^{s}(f_2;e_2,e_2,f_2)=J^{s}(f_2;e_2,e_3,f_2)=0$ and condition \eqref{eigenj31}, we deduce the following system:
\begin{equation} \label{sysJ33}
\begin{alignedat}{2}
\delta_{3}\begin{pmatrix}\omega_{1}\\ \omega_{2}\\ \omega_{3}\end{pmatrix}&=0, &\qquad \delta_{1}\omega_{2}&=0, \\
\delta_{2}\omega_{1}&=0, &\qquad \delta_{1}\omega_{1}+\delta_{2}\omega_{2}+\delta_{3}\omega_{3}&=0.
\end{alignedat}
\end{equation}

 We separate as follows:
\begin{enumerate}[1)]
\item $\delta=0$. 
\begin{enumerate}[a)]
\item If $\omega=0$, then $J$ is isomorphic to $(3,2)_{15}$.
   
    \item If $\omega_{1}\omega_{2}=0$, we have two possibilities:
    \begin{enumerate}[i)]
        \item If $\omega_{k}\neq0$ for some $k\in\{1,2\}$, by changing $e_{1}^{\prime}\mapsto \omega_{1}e_{1}+\omega_{2}e_{2}+\omega_{3}e_{3}, \ e_{2}^{\prime}\mapsto \omega_{2}e_{1}+\omega_{1}e_{2}, \ e_{3}^{\prime}\mapsto \omega_{k}^{2}e_{3}$ we obtain $J\cong (3,2)_{16}$.
        \item If $\omega_{1}=\omega_{2}=0$, then we may assume $\omega_{3}\neq 0$ and by replacing $f_{2}^{\prime}\mapsto \omega_{3}^{-1}f_{2}$, we have $J\cong(3,2)_{17}$.
    \end{enumerate}
     \item If $\omega_1\omega_2\neq0$, by changing $e_{1}^{\prime}\mapsto \omega_{1}e_{1}, \ e_{2}^{\prime}\mapsto\omega_{2}e_{2}+\omega_{3}e_{3}, \ e_{3}\mapsto \omega_{1}\omega_{2}e_{3}$, we obtain $J\cong (3,2)_{18}$. 
\end{enumerate}
\item $\delta\neq0$. 
\begin{enumerate}[a)]
    \item Assume that $\delta_{3}\neq 0$. Then from \eqref{sysJ33}, $\omega=0$. In this case, by changing $e_{1}^{\prime}\mapsto \delta_{3}^{-1}e_{1}-\delta_{1}\delta_{3}^{-2}e_3, \ e_{2}^{\prime}\mapsto e_{2}-\delta_{2}\delta_{3}^{-1}e_{3}, \ e_{3}^{\prime}\mapsto \delta_{3}^{-1}e_{3}$, we obtain $J\cong (3,2)_{19}$.
    \item Assume that $\delta_{3}=0$ and $\delta_{1}\neq 0$. Then from \eqref{sysJ33} we get $\omega_{1}=\omega_{2}=0$, while $\delta_{2}$ and $\omega_{3}$ are arbitrary. We have the following possibilities:
    \begin{enumerate}[i)]
        \item If $\delta_{2}=\omega_{3}=0$, by replacing $f_{1}^{\prime}\mapsto\delta_{1}f_{1}$, we have $J\cong (3,2)_{20}$.
         
        \item If $\delta_{2}=0$ and $\omega_{3}\neq0$, by changing $ f_{1}^{\prime}\mapsto\delta_{1}(\sqrt{\delta_{1}\omega_{3}})^{-1}f_{1}, \ f_{2}^{\prime}\mapsto(\sqrt{\delta_{1}\omega_{3}})^{-1}f_{2}$, we obtain $J\cong (3,2)_{21}$.  
         \item If $\delta_{2}\neq 0$ and $\omega_{3}=0$, by changing $e_{1}^{\prime}\mapsto \delta_{1}^{-1}e_{1}, \ e_{2}^{\prime}\mapsto \delta_{2}^{-1}e_{2}, \ e_{3}^{\prime}\mapsto (\delta_{1}\delta_{2})^{-1}e_{3}$, we obtain $J\cong(3,2)_{22}$.
        \item If $\delta_{2}\omega_{3}\neq0$, by changing $e_{1}^{\prime}\mapsto \delta_{2}\omega_{3}e_{1}, \ e_{2}^{\prime}\mapsto \delta_{1}\omega_{3}e_{2}, \ e_{3}^{\prime}\mapsto \delta_{1}\delta_{2}\omega_{3}^{2}e_{3}, \ f_{1}^{\prime}\mapsto \delta_{1}\delta_{2}\omega_{3}f_{1}$, we obtain $J\cong(3,2)_{23}$.
    \end{enumerate}
    \item Assume $\delta_{3}=\delta_{1}=0$ and $\delta_{2}\neq 0$. From \eqref{sysJ33}, we get $\omega_{1}=\omega_{2}=0$ and $\omega_{3}$ is arbitrary.
    \begin{enumerate}[i)]
        \item If $\omega_{3}=0$, by changing $e_{1}^{\prime}\mapsto e_{2}, \ e_{2}^{\prime}\mapsto e_{1}, \ f_{1}^{\prime}\mapsto\delta_{2}f_{1}$, we obtain $J\cong (3,2)_{20}$.
        
        \item If $\omega_{3}\neq 0$, by changing $e_{1}^{\prime}\mapsto e_{2}, \ e_{2}^{\prime}\mapsto \delta_{2}\omega_{3}e_{1}, \ e_{3}^{\prime}\mapsto \delta_{2}\omega_{3}e_{3}, \ f_{1}^{\prime}\mapsto\delta_{2}f_{1}$, we obtain $J\cong (3,2)_{21}$.   
    \end{enumerate}
\end{enumerate}
\end{enumerate}
\underline{$J_{0}=J_{3,4}$:} The conditions $J^{s}(w;e_{1},e_{1},f_{2})=0$, with $w=e_{1},f_{1},f_{2}$, $J^{s}(f_{2};x,e_{2},f_{2})=0$, with $x=e_{1},e_{2},e_{3}$, and $J^{s}(f_2;e_1,e_1,f_1)=J^{s}(f_2;e_1,e_3,f_2)=0$, together with the condition \eqref{eigenj31} lead us to the following system: 
\begin{equation} \label{sysJ34}
\begin{alignedat}{2}
\delta _3&=0, &\qquad \omega_1&=0,\\ 
     \delta_2 \omega_2&=0, &\qquad     2\delta_1\omega_2-\delta_2\omega_3&=0. 
      \end{alignedat}
\end{equation}

\begin{enumerate}[1)]
\item $\delta=0$.
    \begin{enumerate}[a)]
    \item If $\omega=0$, then $J$ is isomorphic to $(3,2)_{24}$.
    \item If $\omega_{2}\neq 0$, by changing $e_{1}^{\prime}\mapsto e_{1}+\frac{1}{2}\omega_{3}\omega_{2}^{-1}e_{2}, \ e_{2}^{\prime}\mapsto e_{2}+\omega_{3}\omega_{2}^{-1}e_{3}, \ f_{2}^{\prime}\mapsto \omega_{2}^{-1}f_{2}$, we obtain $J\cong(3,2)_{25}$.
    \item If $\omega_{2}=0$, then $\omega_{3}\neq 0$.  By replacing $f_{2}^{\prime}\mapsto\omega_{3}^{-1}f_{2}$, we have $J\cong(3,2)_{26}$.
    \end{enumerate}
    \item $\delta\neq 0$.
    \begin{enumerate}[a)]
      \item Assume that $\delta_2\neq 0$. From \eqref{sysJ34} we get $\omega=0$. By changing $e_{1}^{\prime}\mapsto e_{1}-\delta_{1}\delta_{2}^{-1}e_{2}, \ e_{2}^{\prime}\mapsto e_{2}-2\delta_{1}\delta_{2}^{-1}e_{3}, \ f_{1}^{\prime}\mapsto \delta_{2}f_{1}$, we obtain $J\cong(3,2)_{27}$.
        \item Assume that $\delta_2=0$ and $\delta_1\neq0$. From \eqref{sysJ34}, we obtain $\omega_{2}=0$ and $\omega_{3}$ is arbitrary. 
      \begin{enumerate}[i)]
          \item If $\omega_{3}=0$, by replacing $f_{1}^{\prime}\mapsto \delta_{1}f_{1}$, we have $J\cong(3,2)_{28}$.
          \item If $\omega_{3}\neq0$, by changing $f_{1}^{\prime}\mapsto \sqrt{\delta_{1}\omega_{3}^{-1}}f_{1}, \ f_{2}^{\prime}\mapsto (\sqrt{\delta_{1}\omega_{3}})^{-1}f_{2}$, we obtain $J\cong (3,2)_{29}$.
      \end{enumerate}

    \end{enumerate}
    
    \end{enumerate}

\end{proof}
 
\subsection{Type $(2,3)$}
\begin{theorem}
 Let $J=J_{0}\oplus J_{1}$ be a nilpotent Jordan superalgebra of type $(2,3)$. Then $J$ is isomorphic to one of the superalgebras described in Table \ref{table:caso23}.  
\end{theorem}

{\footnotesize
\begin{longtable}{|l|l|l|l|l|l|}
\caption{Nilpotent Jordan superalgebras of type $(2,3)$}\label{table:caso23}\\ \hline
$J$&\text{multiplication table}&\text{Obs} & \text{nilindex} &$\mathrm{Aut}(J)$& $\mathrm{Ann}(J) $ \\ \hline
$ (2,3)_{1} $ & $ e_{i}e_{j}=0, \ e_{i}f_{k}=0, \ f_{k}f_{l}=0, \ i,j\in \{1,2\}, \ k,l \in\{1,2,3\}.$ & $\text{A}$ & $2$ & $13$ & $ (2,3) $ \\ \hline
$ (2,3)_{2} $ & $ f_{1}f_{2}=e_{1} $. & $\text{A}$ & $3$ & $9$ & $ (2,1) $ \\ \hline
$ (2,3)_{3} $ & $ f_{1}f_{2}=e_{1}, \ f_{1}f_{3}=e_{2} $. & $\text{A}$ & $3$ & $7$ & $ (2,0) $ \\ \hline
$ (2,3)_{4} $ & $ e_{1}f_{2}=f_{1} $. & $\text{A}$ & $3$ & $8$ & $ (1,2) $ \\ \hline
$ (2,3)_{5} $ & $ e_{1}f_{2}=f_{1}, \ f_{1}f_{2}=e_{2} $. & $\text{NA}$ & $4$ & $6$ & $ (1,1) $ \\ \hline
$ (2,3)_{6} $ & $ e_{1}f_{2}=f_{1}, \ f_{2}f_{3}=e_{2} $. & $\text{A}$ & $3$ & $7$ & $ (1,1) $ \\\hline
$ (2,3)_{7} $ & $ e_{1}f_{2}=f_{1}, \ f_{1}f_{3}=e_{2} $. & $\text{NA}$ & $4$ & $5$ & $ (1,0) $ \\ \hline
$ (2,3)_{8} $ & $ e_{1}f_{2}=f_{1}, \ f_{2}f_{3}=e_{1} $. & $\text{NA}$ & $4$ & $6$ & $ (1,1) $ \\ \hline
$ (2,3)_{9} $ & $ e_{1}f_{2}=f_{1}, \ f_{1}f_{2}=e_{2}, \ f_{2}f_{3}=e_{1} $. & $\text{NA}$ & $5$ & $5$ & $ (1,0) $ \\ \hline
$ (2,3)_{10} $ & $ e_{1}f_{2}=f_{1}, \ f_{1}f_{3}=e_{2}, \ f_{2}f_{3}=e_{1} $. & $\text{NA}$ & $5$ & $4$ & $ (1,0) $ \\ \hline
$ (2,3)_{11} $ & $ e_{1}f_{2}=f_{1}, \ e_{2}f_{3}=f_{1} $. & $\text{A}$ & $3$ & $7$ & $ (0,1) $ \\ \hline
$ (2,3)_{12} $ & $ e_{1}f_{2}=f_{1}, \ e_{2}f_{3}=f_{1}, \ f_{2}f_{3}=e_{2} $. & $\text{NA}$ & $4$ & $5$ & $ (0,1) $ \\ \hline
$ (2,3)_{13} $ & $ e_{1}f_{2}=f_{1}, \ e_{2}f_{2}=f_{3} $. & $\text{A}$ & $3$ & $7$ & $ (0,2) $ \\ \hline
$ (2,3)_{14} $ & $ e_{1}f_{2}=f_{1}, \ e_{2}f_{2}=f_{3}, \ f_{1}f_{2}=e_{2} $. & $\text{NA}$ & $5$ & $5$ & $ (0,1) $ \\ \hline
$ (2,3)_{15} $ & $ e_{1}f_{2}=f_{1}, \ e_{2}f_{3}=f_{2} $. & $\text{NA}$ & $4$ & $4$ & $ (0,1) $ \\ \hline
$ (2,3)_{16} $ & $ e_{1}f_{2}=f_{1}, \ e_{1}f_{3}=f_{2} $. & $\text{NA}$ & $4$ & $6$ & $ (1,1) $ \\ \hline

$ (2,3)_{17} $ & $ e_{1}f_{2}=f_{1}, \ e_{1}f_{3}=f_{2}, \ f_{2}f_{3}=e_{2} $. & $\text{NA}$ & $4$ & $5$ & $ (1,1) $ \\ \hline

$ (2,3)_{18} $ & $ e_{1}f_{2}=f_{1}, \ e_{1}f_{3}=f_{2}, \ f_{1}f_{3}=e_{2} $. & $\text{NA}$ & $5$ & $4$ & $ (1,0) $ \\ \hline

$ (2,3)_{19} $ & $ e_{1}f_{2}=e_{2}f_{3}=f_{1}, \ e_{1}f_{3}=f_{2} $. & $\text{NA}$ & $4$ & $5$ & $ (0,1) $ \\ \hline
$ (2,3)_{20} $ & $ e_{1}f_{2}=e_{2}f_{3}=f_{1}, \ e_{1}f_{3}=f_{2}, \ f_{2}f_{3}=e_{2} $. & $\text{NA}$ & $5$ & $4$ & $ (0,1) $ \\ \hline\hline
$ (2,3)_{21} $ & $ e_{1}^{2}=e_{2} $. & $\text{A}$ & $3$ & $11$ & $ (1,3) $ \\ \hline
$ (2,3)_{22} $ & $ e_{1}^{2}=e_{2}, \ f_{1}f_{2}=e_{1} $. & $\text{NA}$ & $5$ & $7$ & $ (1,1) $ \\ \hline
$ (2,3)_{23} $ & $ e_{1}^{2}=e_{2}, \ f_{1}f_{2}=e_2 $. & $\text{A}$ & $3$ & $8$ & $ (1,1) $ \\ \hline
$ (2,3)_{24} $ & $ e_1^{2}=e_2, \ f_{1}f_{2}=e_1, \ f_{1}f_{3}=e_2  $. & $\text{NA}$ & $5$ & $5$ & $ (1,0) $ \\ \hline
$ (2,3)_{25} $ & $ e_{1}^{2}=e_{2}, \ e_{1}f_{2}=f_{1} $. & $\text{A}$ & $3$ & $7$ & $ (1,2) $ \\ \hline
$ (2,3)_{26} $ & $ e_{1}^{2}=e_{2}, \ e_{1}f_{2}=f_{1}, \ f_{2}f_{3}=e_{2} $. & $\text{A}$ & $3$ & $6$ & $ (1,1) $ \\ \hline
$ (2,3)_{27} $ & $ e_{1}^{2}=e_{2}, \ e_{1}f_{2}=f_{1}, \ f_{1}f_{2}=e_{2} $. & $\text{NA}$ & $4$ & $5$ & $ (1,1) $ \\ \hline
$ (2,3)_{28} $ & $ e_{1}^{2}=e_{2}, \ e_{1}f_{2}=f_{1}, \ f_{1}f_{3}=e_{2} $. & $\text{NA}$ & $4$ & $4$ & $ (1,0) $ \\ \hline
$ (2,3)_{29} $ & $ e_{1}^{2}=e_{2}, \ e_{1}f_{2}=f_{1}, \ f_{2}f_{3}=e_{1} $. & $\text{NA}$ & $5$ & $5$ & $ (1,1) $ \\ \hline
$ (2,3)_{30} $ & $ e_{1}^{2}=e_{2}, \ e_{1}f_{2}=f_{1}, \ f_{1}f_{2}=e_{2}, \ f_{2}f_{3}=e_{1} $. & $\text{NA}$ & $5$ & $4$ & $ (1,0) $ \\ \hline

$ (2,3)_{31}^{\lambda} $ & $ e_{1}^{2}=e_{2}, \ e_{1}f_{2}=f_{1}, \ f_{1}f_{3}=\lambda e_{2}, \ f_{2}f_{3}=e_{1}, \  \text{\footnotesize{$\lambda\in \mathbb{C}^{\ast}$}} $. & $\text{NA}$ & $5$ & $4$ & $ (1,0) $ \\ \hline
$ (2,3)_{32} $ & $ e_{1}^{2}=e_{2}, \ e_{1}f_{2}=f_{1}, \ e_{1}f_{3}=f_{2} $. & $\text{NA}$ & $4$ & $5$ & $ (1,1) $ \\ \hline
$ (2,3)_{33} $ & $ e_{1}^{2}=e_{2}, \ e_{1}f_{2}=f_{1}, \ e_{1}f_{3}=f_{2}, \ f_{1}f_3=e_{2} $. & $\text{NA}$ & $5$ & $3$ & $ (1,0) $ \\ \hline
$ (2,3)_{34} $ & $ e_{1}^{2}=e_{2}, \ e_{1}f_{2}=f_{1}, \ e_{1}f_{3}=f_{2}, \ f_{2}f_{3}=e_{2} $. & $\text{NA}$ & $4$ & $4$ & $ (1,1) $ \\ \hline
$ (2,3)_{35} $ & $ e_{1}^{2}=e_{2}, \ e_{2}f_{2}=f_1 $. & $\text{NA}$ & $4$ & $6$ & $ (0,2) $ \\ \hline
$ (2,3)_{36} $ & $ e_{1}^{2}=e_{2}, \ e_{2}f_{2}=f_1, \ f_{2}f_{3}=e_{2} $. & $\text{NA}$ & $4$ & $5$ & $ (0,1) $ \\ \hline
$ (2,3)_{37} $ & $ e_{1}^{2}=e_{2}, \ e_{2}f_{2}=f_{1}, \ f_{2}f_{3}=e_{1} $. & $\text{NA}$ & $6$ & $5$ & $ (0,1) $ \\ \hline

$ (2,3)_{38} $ & $ e_{1}^{2}=e_{2}, \ e_{1}f_{2}=e_{2}f_{2}=f_{1}, \ f_{2}f_{3}=e_{1} $. & $\text{NA}$ & $6$ & $4$ & $ (0,1) $ \\ \hline
$ (2,3)_{39} $ & $ e_{1}^{2}=e_{2}, \ e_{1}f_{3}= e_{2}f_{2}=f_{1} $. & $\text{NA}$ & $4$ & $5$ & $ (0,1) $ \\ \hline
$ (2,3)_{40} $ & $ e_{1}^{2}=e_{2}, \ e_{1}f_{3}=e_{2}f_{2}=f_1, \ f_{2}f_{3}=e_2 $. & $\text{NA}$ & $4$ & $4$ & $ (0,1) $ \\ \hline
$ (2,3)_{41} $ & $ e_{1}^{2}=e_{2}, \ e_{1}f_{3}=e_{2}f_{2}=f_{1}, \ f_{2}f_{3}=e_{1} $. & $\text{NA}$ & $6$ & $3$ & $ (0,1) $ \\ \hline 

$ (2,3)_{42} $ & $ e_{1}^{2}=e_{2}, \ e_{1}f_{2}=f_{3}, \ e_{2}f_{2}=f_{1} $. & $\text{NA}$ & $4$ & $5$ & $ (0,2) $ \\ \hline
$ (2,3)_{43}^{1} $ & $ e_{1}^{2}=e_2, \ e_{1}f_{2}=f_{3}, \ e_{1}f_{3}=f_{1}, \ e_{2}f_{2}=f_{1}$.  & $\text{A}$ & $4$ & $5$ & $ (0,1) $ \\ \hline
$ (2,3)_{43}^{\gamma} $ & $ e_{1}^{2}=e_2, \ e_{1}f_{2}=f_{3}, \ e_{1}f_{3}=\gamma f_{1}, \ e_{2}f_{2}=f_{1}, \text{\footnotesize{ $\gamma\in \mathbb{C}^{\ast} \ \backslash \ \{1\}$}} $. & $\text{NA}$ & $4$ & $5$ & $ (0,1) $ \\ \hline
$ (2,3)_{44}^{\phi} $ & $ e_1^{2}=e_2, \ e_{1}f_{2}=f_{3}, \ e_{1}f_{3}=\phi f_{1}, \ e_{2}f_{2}=f_{1}, \ f_{2}f_{3}=e_{2},\text{\footnotesize{ $\phi\in \mathbb{C}$}}$. & $\text{NA}$ & $5$ & $4$ & $ (0,1) $ \\ \hline

\end{longtable}
}
These superalgebras are pairwise non-isomorphic except for
\begin{itemize}
    \item $(2,3)_{31}^{\lambda_{1}}\cong(2,3)_{31}^{\lambda_{2}}$ if and only if $\lambda_{1}=\lambda_{2}$;
    \item $(2,3)_{43}^{\gamma_{1}}\cong (2,3)_{43}^{\gamma_{2}}$ if and only if $\gamma_{1}=\gamma_{2}$;
    \item $(2,3)_{44}^{\phi_{1}}\cong(2,3)_{44}^{\phi_{2}}$ if and only if $\phi_{1}=\phi_{2}$.
\end{itemize}
\begin{proof}
Let $J=J_{0}\oplus J_{1}$ be a nilpotent Jordan superalgebra of type $(2,3)$. There are two non-isomorphic possibilities for $J_{0}$, namely the nilpotent Jordan algebras $J_{2,1}$ and $J_{2,2}$ in Table \ref{table:algjor4}. We first consider the trivial case.

\underline{$J_{0}=J_{2,1}$}: Since fewer restrictions are imposed in this case, the strategies employed previously do not yield a suitable simultaneous representation for $l_{e_{1}}$ and $l_{e_{2}}$.
Therefore, we follow the strategy presented in \cite[Section 2]{FormBurde}. Since $l_{e_{1}}$ and $l_{e_{2}}$ are simultaneously triangularizable, we may choose a basis $\mathcal{B}_{1}$ of $J_{1}$ such that:

\begin{equation}\label{burde}[l_{e_{1}}]_{\mathcal{B}_{1}}=\left(\begin{array}{ccc}
0 &\lambda_1     &\lambda_{2}  \\
0     & 0 & \lambda_3\\
0&0&0\end{array}\right), \ \text{and } [l_{e_{2}}]_{\mathcal{B}_{1}}=\left(\begin{array}{ccc}
0 &\delta_{3}    &\delta_{4}  \\
0     & 0 & \delta_{5}\\
0&0&0\end{array}\right),\end{equation}

where $\lambda_{i}\in\{0,1\}$, $\delta_i \in \mathbb{C}$, and each row and each column of $[l_{e_{1}}]_{\mathcal{B}_{1}}$ there is at most one nonzero entry. We denote $f_{i}f_{j}=\omega_{ij}e_1+\theta_{ij}e_2$. We split our analysis by considering all possible forms of $[l_{e_{1}}]_{\mathcal{B}_{1}}$, according to \eqref{burde}:
\[
\begin{array}{ccccc}
\matitem{\begin{pmatrix}0&0&0\\0&0&0\\0&0&0\end{pmatrix}}, &
\matitem{\begin{pmatrix}0&1&0\\0&0&0\\0&0&0\end{pmatrix}}, &
\matitem{\begin{pmatrix}0&0&1\\0&0&0\\0&0&0\end{pmatrix}}, &
\matitem{\begin{pmatrix}0&0&0\\0&0&1\\0&0&0\end{pmatrix}}, &
\matitem{\begin{pmatrix}0&1&0\\0&0&1\\0&0&0\end{pmatrix}}.
\end{array}\]

Moreover, since $J_{2,1}$ has a trivial multiplication, we may interchange the roles of $e_{1}$ and $e_{2}$, if necessary, and assume that the ranks of the linear maps satisfy  $\operatorname{rk}(l_{e_2}) \leq \operatorname{rk}(l_{e_1})$.

\begin{enumerate}[1)]
\item $l_{e_1}=0$. We can assume $l_{e_2}=0$. In this situation, there are no restrictions on the parameters $\omega_{ij}$ and $\theta_{ij}$. Therefore, $(\omega_{ij})$ and $(\theta_{ij})$ are arbitrary skew-symmetric bilinear forms and by Lemma \ref{iso14} our problem is reduced to that of determining all possible pairs of such forms. By the Jordan–Kronecker Theorem \cite[Theorem1]{Kronecker} we can find a basis of $J_{1}$ such that $(\omega_{ij})$ and $(\theta_{ij})$ are simultaneously congruent to one of the following pairs $(\omega_{ij})^{\prime}$ and $(\theta_{ij})^{\prime}$, respectively. 

\begin{enumerate}[a)]
    \item $(\omega_{ij})^{\prime}=0$ and $(\theta_{ij})^{\prime}=0$. In this case, $J\cong(2,3)_{1}$.
    \item $(\omega_{ij})^{\prime}=\left(\begin{array}{ccc}
  0&1   & 0 \\
  -1 & 0 &0\\
  0&0&0
\end{array}\right)$ and $(\theta_{ij})^{\prime}=0$. Thus $J\cong(2,3)_{2}$.
\item $(\omega_{ij})^{\prime}=\left(\begin{array}{ccc}
  0&\lambda   & 0 \\
  -\lambda & 0 &0\\
  0&0&0
\end{array}\right)$, $\lambda\neq0$, and $(\theta_{ij})^{\prime}=\left(\begin{array}{ccc}
  0&1   & 0 \\
  -1 & 0 &0\\
  0&0&0
\end{array}\right)$. By replacing  $e_{1}^{\prime}\mapsto\lambda e_{1}+e_{2}$, we have $J\cong(2,3)_{2}$.
\item $(\omega_{ij})^{\prime}=0$ and $(\theta_{ij})^{\prime}=\left(\begin{array}{ccc}
  0&1   & 0 \\
  -1 & 0 &0\\
  0&0&0
\end{array}\right)$. By changing the roles of $e_{1}$ and $e_{2}$, we have $J\cong (2,3)_{2}$.
\item $(\omega_{ij})^{\prime}=\left(\begin{array}{ccc}
  0&1   & 0 \\
  -1 & 0 &0\\
  0&0&0
\end{array}\right)$ and $(\theta_{ij})^{\prime}=\left(\begin{array}{ccc}
  0&0   & 1 \\
  0 & 0 &0\\
  -1&0&0
\end{array}\right)$. Then $J\cong (2,3)_{3}$.
\end{enumerate}

\item \label{casej21rank1a12}   $[l_{e_{1}}]_{\mathcal{B}_{1}}=\left(\begin{array}{ccc}
0 &1     &0  \\
0     & 0 & 0\\
0&0&0\end{array}\right)$: Since we can assume that $\mathrm{rank}(l_{e_{2}})\leq1$, we must have $\delta_{3}\delta_{5}=0$
\begin{enumerate}[a)]
    \item \label{delta50} If $\delta_{5}=0$. If necessary, we may take a suitable change of basis in $J_{0}$ and we obtain $[l_{e_{2}}]_{\mathcal{B}_{1}}=\left(\begin{array}{ccc}
0 &0     &\delta_{4}  \\
0     & 0 & 0\\
0&0&0\end{array}\right)$. From $J^{s}(f_{2};e_{1},f_{2},f_{3})= J^{s}(f_{3};e_{2},f_{2},f_{3})=0$ and the nilpotency of $L_{f_{2}}$ and $L_{f_{2}+f_{3}}$, we get the following system:
\begin{equation}\label{sysJ21rank1}
\begin{alignedat}{2}
\omega_{12}&=0, &\qquad \omega_{13}&=0, \\
\delta_{4}\begin{pmatrix}\theta_{12}\\ \theta_{13}\end{pmatrix} &= 0.
\end{alignedat}    
\end{equation}

\begin{enumerate}[i)]
\item If $\delta_{4}=0$. Then $\omega_{23},\theta_{12},\theta_{13}$ and $\theta_{23}$ can assume arbitrary values in $\mathbb{C}$.
\begin{enumerate}[A)]
    \item If $(\omega_{23},\theta_{12},\theta_{13},\theta_{23})=0$, thus $J\cong (2,3)_{4}$.
    \item If $\omega_{23}=\theta_{13}=0$, we have two possibilities. If $\theta_{12}\neq0$, by changing $e_{2}^{\prime}\mapsto \theta_{12}e_{2}, f_{3}^{\prime}\mapsto \theta_{23}f_{1}+\theta_{12}f_{3}$ we obtain $J\cong(2,3)_{5}$. If  $\theta_{12}=0$, we may assume $\theta_{23}\neq 0$ and by replacing $e_{2}^{\prime}\mapsto \theta_{23}e_{2}$ we have $J\cong(2,3)_{6}$.
    
    \item If $\omega_{23}=0$ and $\theta_{13}\neq 0$, by changing $e_{2}^{\prime}\mapsto\theta_{13}e_{2}, \ f_{2}^{\prime}\mapsto-\theta_{23}\theta_{13}^{-1}f_{1}+f_{2}-\theta_{12}\theta_{13}^{-1}f_{3}$, we obtain $J\cong (2,3)_{7}$.
    \item If $\omega_{23}\neq0$ and $\theta_{13}=0$ we have two possibilities. If $\theta_{12}=0$, by changing $e_{1}^{\prime}\mapsto\omega_{23}e_{1}+\theta_{23}e_{2}, \ f_{1}^{\prime}\mapsto \omega_{23}f_{1}$ we obtain $J\cong(2,3)_{8}$. If $\theta_{12}\neq0$, by changing $e_{1}^{\prime}\mapsto \omega_{23}e_{1}+\theta_{23}e_{2}, \ e_{2}^{\prime}\mapsto \omega_{23}\theta_{12}e_{2}, \ f_{1}^{\prime}\mapsto\omega_{23}f_{1}$ we get $J\cong(2,3)_{9}$.
    \item If $\omega_{23}\neq0$ and $\theta_{13}\neq0$, by changing $e_{1}^{\prime}\mapsto \omega_{23}e_{1}+\theta_{23}e_{2}, \ e_{2}^{\prime}\mapsto\omega_{23}\theta_{13}e_{2}, \ f_{1}^{\prime}\mapsto\omega_{23}f_{1}, \ f_{2}^{\prime}\mapsto f_{2}-\theta_{12}\theta_{13}^{-1}f_{3}$ we obtain $J\cong (2,3)_{10}$.
\end{enumerate}
    \item If $\delta_{4}\neq0$. Then from \eqref{sysJ21rank1} we get $\theta_{12}=\theta_{13}=0$ and $\omega_{23},\theta_{23}\in\mathbb{C}$.
    \begin{enumerate}[A)]
    \item If  $(\omega_{23},\theta_{23})=0$, by replacing $f_{i}^{\prime}\mapsto\delta_{4}f_{i}$, with $i=1,2$, we have $J\cong(2,3)_{11}$.
     \item If $(\omega_{23},\theta_{23})\neq 0$, then $J$ is isomorphic to $(2,3)_{12}$ via the change of basis $e_{1}\mapsto-\omega_{23}\delta_{4}^{-1}e_{2}, \ e_{2}^{\prime}\mapsto -\omega_{23}e_{1}, \ f_{1}^{\prime}\mapsto-\omega_{23}f_{1}, \ f_{2}^{\prime}\mapsto f_{3}, \ f_{3}^{\prime}\mapsto f_{2}$, if $\theta_{23}= 0$, or $e_{1}^{\prime}\mapsto \delta_{4}\theta_{23}e_{1}, \ e_{2}^{\prime}\mapsto -\omega_{23}\delta_{4}\theta_{23}e_{1}-\delta_{4}\theta_{23}^{2}e_{2}, \ f_{1}^{\prime}\mapsto-\delta_{4}^{2}\theta_{23}^{2}f_{1}, \ f_{2}^{\prime}\mapsto -\delta_{4}\theta_{23}f_{2}+\omega_{23}f_{3}$, if $\theta_{23}\neq0$.
    \end{enumerate}
     
\end{enumerate}
   \item If $\delta_{5}\neq 0$. By replacing $e_{1}^{\prime}\mapsto e_{1}+\delta_{5}^{-1}e_{2}$, we obtain that $l_{e_{1}^{\prime}}$ has rank 2, and this case will be considered in \eqref{caseJ21rank2}.
\end{enumerate}
\item $[l_{e_{1}}]_{\mathcal{B}_{1}}=\left(\begin{array}{ccc}
0 &0     &1  \\
0     & 0 & 0\\
0&0&0\end{array}\right)$: As in the previous case, we must have $\delta_{3}\delta_{5}=0$.
\begin{enumerate}[a)]
    \item If $\delta_{5}=0$. Changing the roles of $f_{2}$ and $f_{3}$, we obtain the case studied in \eqref{delta50}.
    \item If $\delta_{5}\neq 0$. Taking the change of basis given by $e_{2}^{\prime}\mapsto -\delta_{4}\delta_{5}^{-1}e_{1}+\delta_{5}^{-1}e_{2}$, if necessary, we may suppose $[l_{e_{2}}]_{\mathcal{B}_{1}}=\left(\begin{array}{ccc}
0 &0     &0  \\
0     & 0 & 1\\
0&0&0\end{array}\right)$. From $J^{s}(f_{3};f_{1},f_{2},f_{3})=0$, $L_{f_3}^{4}(f_1)=(\omega_{13}^{2}+\omega_{23}\theta_{13})f_1$ and the nilpotency condition applied to the trace of $L_{f_{3}}^{2}$, $L_{f_2+f_3}^{2}$, $L_{f_1+f_3}^{2}$, we obtain the following system:
\begin{equation}
\begin{alignedat}{2}
\omega_{12}&=0, &\qquad \theta_{12}&=0,\\
\omega_{13}&=-\theta_{23}, &\qquad \omega_{13}^{2}+\omega_{23}\theta_{13}&=0. 
\end{alignedat}    
\end{equation}
\begin{enumerate}[i)]
    \item If $\theta_{13}=0$, then $\omega_{13}=\theta_{23}=0$. We split in two cases:
    \begin{enumerate}[A)]
       \item If $\omega_{23}=0$, by changing $f_{2}^{\prime}\mapsto f_{3}, \ f_{3}^{\prime}\mapsto f_{2}$, we obtain $J\cong (2,3)_{13}$.
        \item If $\omega_{23}\neq0$, by changing $e_{1}^{\prime}\mapsto e_{2}, \ e_{2}^{\prime}\mapsto\omega_{23}e_{1}, \ f_{1}^{\prime}\mapsto f_{2}, \ f_{2}^{\prime}\mapsto f_{3}, \ f_{3}^{\prime}\mapsto\omega_{23}f_{1}$, we get $J\cong(2,3)_{14}$.
    \end{enumerate}
    \item If $\theta_{13}\neq0$, then $\omega_{23}=-\omega_{13}^{2}\theta_{13}^{-1}$. Then $J$ is isomorphic to $J\cong(2,3)_{14}$ via the change of basis $e_{2}^{\prime}\mapsto \omega_{13}e_{1}+\theta_{13}e_{2}, \ f_{2}^{\prime}\mapsto f_{3}, \ f_{3}^{\prime}\mapsto \omega_{13}f_{1}+\theta_{13}f_{2}$.
\end{enumerate}
\end{enumerate}

    \item $[l_{e_{1}}]_{\mathcal{B}_{1}}=\left(\begin{array}{ccc}
0 &0     &0  \\
0     & 0 & 1\\
0&0&0\end{array}\right)$. A straightforward computation shows that, according to the values of $\delta_{3},\delta_{4}$ and $\delta_{5}$, we may perform a change of basis so that $l_{e_{1}}$ is similar to one of the forms (2), (3), or (5).

    \item \label{caseJ21rank2}$[l_{e_{1}}]_{\mathcal{B}_{1}}=\left(\begin{array}{ccc}
0 &1     &0  \\
0     & 0 & 1\\
0&0&0\end{array}\right)$. The Jordan superidentity in the basis elements $J^{s}(e_1;e_1,f_2,f_3)$, $J^{s}(f_3;e_1,f_2,f_3)$, $J^{s}(f_3,e_2,f_2,f_3)$ 
 provides the following conditions:

\begin{equation*}
\begin{alignedat}{2}
\omega_{12}&=0, &\qquad \theta_{12}&=0, \\
-2 \omega _{23}+\theta _{13} \delta _4-\theta _{23} \left(\delta _3+\delta _5\right)&=0, & \qquad \omega _{13}+\theta _{13} \delta _5&=0, \\
\left(\delta _3-\delta _5\right) \left(\theta _{23} \left(\delta _3-\delta _5\right)-\theta _{13} \delta _4\right)&=0.
\end{alignedat}
\end{equation*}
Under such conditions, we also obtain from the nilpotency of $L_{f_{3}}$ and $L_{f_{2}+f_{3}}$ that $\theta_{13}\delta_{4}=0$ and $\theta_{13}(\delta_{5}-\delta_{3})=0$, respectively, since $\mathrm{tr}(L_{f_{3}}^{2})=-3\theta _{13} \delta _4$ and $\mathrm{tr}(L_{f_{2}+f_{3}}^{2})=\theta _{13} \left(-2 \delta _3-3 \delta _4+2 \delta _5\right)$.

So we need to analyze all the possible solutions of the system:
\begin{equation}\label{sysJ21rank2}
\begin{alignedat}{2}
\omega_{12}&=0, &\qquad \theta_{12}&=0, \\
\omega _{23}+\theta _{23}\delta _5&=0, & \qquad \omega _{13}+\theta _{13} \delta _5&=0, \\
(\delta_{5}-\delta_{3})\begin{pmatrix}\theta_{13}\\ \theta_{23}\end{pmatrix} &= 0,& \qquad \theta_{13}\delta_{4}&=0.\\
\end{alignedat}
\end{equation}
\begin{enumerate}[a)]
\item If $\delta_3\neq \delta_{5}$. Then from \eqref{sysJ21rank2} we get $\theta_{13}=\theta_{23}=\omega_{13}=\omega_{23}=0$ and $\delta_{4}$ is arbitrary. In this case, $J$ is isomorphic to $(2,3)_{15}$ via the change of basis $e_{1}^{\prime}\mapsto-\delta_{5}e_{1}+e_{2}, \ e_{2}^{\prime}\mapsto-\delta_{3}e_{1}+e_{2}, \ f_{1}^{\prime}\mapsto(\delta_{3}-\delta_{5})f_{1}, \ f_{2}^{\prime}\mapsto -\delta_{4}(\delta_{3}-\delta_{5})^{-1}f_{1}+f_{2}, \ f_{3}^{\prime}\mapsto \delta_{4}(\delta_{3}-\delta_{5})^{-2}f_{2}-(\delta_{3}-\delta_{5})^{-1}f_{3}$.

\item If $\delta_{3}=\delta_{5}$ and $\delta_{4}=0$. Then from \eqref{sysJ21rank2} we get $\omega_{13}=-\theta_{13}\delta_{3}$ and $\omega_{23}=-\theta_{23}\delta_{3}$, while $\theta_{13},\theta_{23}$ and $\delta_{3}$ remain free.
\begin{enumerate}[i)]
    \item If $\theta_{13}=\theta_{23}=0$, by replacing $e_{2}^{\prime}\mapsto -\delta_{3}e_{1}+e_{2}$, we have $J\cong (2,3)_{16}$
    \item If $\theta_{13}=0$ and $\theta_{23}\neq 0$, by changing $e_{2}^{\prime}\mapsto-\theta_{23}\delta_{3}e_{1}+\theta_{23}e_{2}$, we obtain $J\cong(2,3)_{17}$.
    \item If $\theta_{13}\neq 0$, by changing  $e_{2}^{\prime}\mapsto -\theta_{13}\delta_{3}e_{1}+\theta_{13}e_{2}, \ f_{2}^{\prime}\mapsto-\theta_{23}\theta_{13}^{-1}f_{1}+f_{2}, \ f_{3}^{\prime}\mapsto -\theta_{23}\theta_{13}^{-1}f_{2}+f_{3}$, we get $J\cong(2,3)_{18}$.
\end{enumerate}
\item $\delta_{3}= \delta_{5}$ and $\delta_{4}\neq 0$. Then from \eqref{sysJ21rank2} we get $\theta_{13}=\omega_{13}=0$ and $\omega_{23}=-\theta_{23}\delta_{3}$, whereas $\theta_{23}$ and $\delta_{3}$ are arbitrary.
 \begin{enumerate}[i)]
  \item If $\theta_{23}=0$, by changing $e_{1}^{\prime}\mapsto\sqrt{\delta_{4}}e_{1}, \ e_{2}^{\prime}\mapsto -\delta_{3}e_{1}+e_{2}, \ f_{1}^{\prime}\mapsto\delta_{4}f_{1}, \ f_{2}^{\prime}\mapsto\sqrt{\delta_{4}}f_{2}$, we obtain $J\cong(2,3)_{19}$.
   \item If  $\theta_{23}\neq0$, by changing $e_{1}^{\prime}\mapsto \theta_{23}\delta_{4}e_{1}, \ e_{2}^{\prime}\mapsto-\theta_{23}^{2}\delta_{4}\delta_{3}e_{1}+\theta_{23}^{2}\delta_{4}e_{2}, \ f_{1}^{\prime}\mapsto \theta_{23}^{2}\delta_{4}^{2}f_1, \ f_{2}^{\prime}\mapsto \theta_{23}\delta_{4}f_{2}$, we get $J\cong(2,3)_{20}$. 
    \end{enumerate}
   
\end{enumerate}

\end{enumerate}

\underline{$J_{0}=J_{2,2}$:} The superidentity $J^{s}(e_1;e_1,e_1,z)=0$, $z\in J_{1}$, gives us $2l_{e_1}^{3}=3l_{e_1}l_{e_2}=3l_{e_2}l_{e_1}$. Since $l_{e_{1}}$ is nilpotent and $\mathrm{dim}\,J_{1}=3$, we have $l_{e_{1}}^{3}=0$. Hence 
\begin{equation}\label{operadorescaso23naotrivial}
l_{e_1}l_{e_2}=l_{e_2}l_{e_1}=0.    
\end{equation}

Let us assume that $\mathcal{B}_1=\{f_{1},f_{2},f_{3}\}$ is such that $[l_{e_{2}}]_{\mathcal{B}_1}$ is in Jordan normal form, that is, $[l_{e_{2}}]_{\mathcal{B}_{1}}=\left(\begin{array}{ccc}
0 &\delta_1     &0  \\
0     & 0 & \delta_2\\
0&0&0
\end{array}\right)$, where $\delta_{1},\delta_{2}\in\{0,1\}$. We keep $f_{i}f_{j}=\omega_{ij}e_1+\theta_{ij}e_2$. Up to matrix similarity, we have the following possibilities:

\begin{enumerate}[1)]
\item  $[l_{e_{2}}]_{\mathcal{B}_{1}}=\left(\begin{array}{ccc}
0 &0     &0  \\
0     & 0 & 0\\
0&0&0
\end{array}\right)$. In this case, with a suitable change of basis $\mathcal{B}_{1}^{\prime}$, if necessary, we can assume that $[l_{e_1}]_{\mathcal{B}^{\prime}_{1}}$ is in Jordan normal form. Up to matrix similarity, we have the following possibilities:
\begin{enumerate}[a)]
\item $[l_{e_1}]_{\mathcal{B}^{\prime}_{1}}=\left(\begin{array}{ccc}
   0&0 &0  \\
  0&0 &0 \\
  0&0&0
\end{array}\right)$: 

In this case, Jordan's superidentity, as well as the nilpotency of $L_{f}$, $f\in J_{1}$, do not provide conditions for the parameters. We use again the Jordan-Kronecker Theorem in order to obtain all non-isomorphic possibilities for this case. 

\begin{enumerate}[i)]
    \item $(\omega_{ij})^{\prime}=0$ and $(\theta_{ij})^{\prime}=0$. In this case, $J\cong(2,3)_{21}$.
    \item $(\omega_{ij})^{\prime}=\left(\begin{array}{ccc}
  0&1   & 0 \\
  -1 & 0 &0\\
  0&0&0
\end{array}\right)$ and $(\theta_{ij})^{\prime}=0$. Thus $J\cong(2,3)_{22}$.
\item  $(\omega_{ij})^{\prime}=\left(\begin{array}{ccc}
  0&\lambda   & 0 \\
  -\lambda & 0 &0\\
  0&0&0
\end{array}\right)$, $\lambda\neq0$, and $(\theta_{ij})^{\prime}=\left(\begin{array}{ccc}
  0&1   & 0 \\
  -1 & 0 &0\\
  0&0&0
\end{array}\right)$. By changing $e_{1}^{\prime}\mapsto \lambda e_{1}+e_{2}, \ e_{2}^{\prime}\mapsto \lambda^{2}e_{2}$, we obtain $J\cong(2,3)_{22}$.
\item $(\omega_{ij})^{\prime}=0$ and $(\theta_{ij})^{\prime}=\left(\begin{array}{ccc}
  0&1   & 0 \\
  -1 & 0 &0\\
  0&0&0
\end{array}\right)$. Thus $J\cong (2,3)_{23}$.
\item $(\omega_{ij})^{\prime}=\left(\begin{array}{ccc}
  0&1   & 0 \\
  -1 & 0 &0\\
  0&0&0
\end{array}\right)$ and $(\theta_{ij})^{\prime}=\left(\begin{array}{ccc}
  0&0   & 1 \\
  0 & 0 &0\\
  -1&0&0
\end{array}\right)$. Then $J\cong (2,3)_{24}$.
\end{enumerate}
\item $[l_{e_1}]_{\mathcal{B}_{1}^{\prime}}=\left(\begin{array}{ccc}
   0&1 &0  \\
  0&0 &0 \\
  0&0&0
\end{array}\right)$: Computing the Jordan superidentity for the basis elements $J^{s}(f_2;e_1,e_1,f_2)$ and $J^{s}(f_3;e_1,e_1,f_2)$ gives the equations  $\omega_{12}=0$ and $\omega_{13}=0$. The nilpotency of $l_f$, $f\in J_{1}$, does not impose further restrictions. Thus, the parameters $\omega_{23}$, $\theta_{12}$, $\theta_{13}$, and $\theta_{23}$ remain free, and we separate into the following cases:
\begin{enumerate}[i)]
    \item $\omega_{23}=\theta_{13}=0$:
    \begin{enumerate}[A)]
        \item If $\theta_{12}=0$ and $\theta_{23}=0$, then $J\cong (2,3)_{25}$. If $\theta_{12}=0$ and $\theta_{23}\neq0$, we have $J\cong(2,3)_{26}$ by replacing $f_{3}^{\prime}\mapsto\theta_{23}^{-1}f_{3}$.
        \item If $\theta_{12}\neq0$, by changing $e_{1}^{\prime}\mapsto\theta_{12}e_{1}, \ e_{2}^{\prime}\mapsto\theta_{12}^{2}e_{2}, \ f_{1}^{\prime}\mapsto\theta_{12}f_{1}, \ f_{3}^{\prime}\mapsto\theta_{23}f_{1}+\theta_{12}f_{3}$, we obtain $J\cong (2,3)_{27}$.
    \end{enumerate}
    \item  $\omega_{23}=0$ and $\theta_{13}\neq0$. In this case, we have that $J$ is isomorphic to $(2,3)_{28}$ via the change of basis $f_{2}^{\prime}\mapsto -\theta_{23}\theta_{13}^{-1}f_{1}+f_{2}-\theta_{12}\theta_{13}^{-1}f_{3}, \ f_{3}^{\prime}\mapsto \theta_{13}^{-1}f_{3}$.
    \item $\omega_{23}\neq0$ and $\theta_{13}=0$:
    \begin{enumerate}[A)]
        \item If $\theta_{12}=0$, by changing $e_{1}^{\prime}\mapsto \omega_{23}e_{1}+\theta_{23}e_{2}, \ e_{2}^{\prime}\mapsto \omega_{23}^{2}e_{2}, \ f_{1}^{\prime}\mapsto \omega_{23}f_{1}$, we obtain $J\cong(2,3)_{29}$.
        \item If $\theta_{12}\neq 0$, by changing $e_{1}^{\prime}\mapsto \theta_{12}e_{1}+\theta_{12}\theta_{23}\omega_{23}^{-1}e_{2}, \ e_{2}^{\prime}\mapsto \theta_{12}^{2}e_{2}, \ f_{1}^{\prime}\mapsto\theta_{12}f_{1},\ f_{3}^{\prime}\mapsto\theta_{12}\omega_{23}^{-1}f_{3}$, we get $J\cong(2,3)_{30}$.
    \end{enumerate}
    \item $\omega_{23}\neq 0$ and $\theta_{13}\neq 0$. By taking $\lambda=\theta_{13}\omega_{23}^{-1}$ and applying the change of basis $e_{1}^{\prime}\mapsto \omega_{23}e_{1}+\theta_{23}e_{2}, \ e_{2}^{\prime}\mapsto \omega_{23}^{2}e_{2}, \ f_{1}^{\prime}\mapsto\omega_{23}f_{1}, \ f_{2}^{\prime}\mapsto f_{2}-\theta_{12}\theta_{13}^{-1}f_{3}$, we obtain $J\cong(2,3)_{31}^{\lambda}$.
\end{enumerate}
    \item $[l_{e_1}]_{\mathcal{B}_{1}^{\prime}}=\left(\begin{array}{ccc}
   0&1 &0  \\
  0&0 &1 \\
  0&0&0
\end{array}\right)$: From $J^{s}(e_1;e_1,f_2,f_3)=J^{s}(f_3;e_1,e_1,f_2)=J^{s}(f_3,e_1,e_1,f_3)=0$, we obtain the following system:
\begin{equation} \label{sysJ232}
\begin{alignedat}{2}
\omega_{12}&=0, &\qquad \theta_{12}&=0, \\
\omega_{13}&=0, &\qquad \omega_{23}&=0.
\end{alignedat}
\end{equation}
As in the previous case, the nilpotency of $l_f$, $f\in J_{1}$, does not impose further restrictions. Thus, the parameters $\theta_{13}$ and $\theta_{23}$ remain free and we separate into the following cases:
\begin{enumerate}[i)]
    \item If $\theta_{13}=\theta_{23}=0$. Thus $J\cong(2,3)_{32}$.
        \item If $\theta_{13}\neq 0$, by changing $f_{1}^{\prime}\mapsto (\sqrt{\theta_{13}})^{-1}f_{1}, \ f_{2}^{\prime}\mapsto -\theta_{23}\theta_{13}^{-3/2}f_{1}+(\sqrt{\theta_{13}})^{-1}f_{2}, \ f_{3}^{\prime}\mapsto -\theta_{23}\theta_{13}^{-3/2}f_{2}+(\sqrt{\theta_{13}})^{-1}f_{3}$, we obtain $J\cong (2,3)_{33}$.
    \item If $\theta_{13}= 0$ and  $\theta_{23}\neq0$, by replacing $f_{i}^{\prime}\mapsto (\sqrt{\theta_{23}})^{-1}f_{i}$, with $i=1,2,3$, we have $J\cong(2,3)_{34}$.

\end{enumerate}

\end{enumerate}

    \item $[l_{e_{2}}]_{\mathcal{B}_1}=\left(\begin{array}{ccc}
0 &1     &0  \\
0     & 0 &0\\
0&0&0
\end{array}\right)$. By \eqref{operadorescaso23naotrivial} and the nilpotency of $l_{e_{1}}$ we have that $[l_{e_1}]_{\mathcal{B}_1}=\left(\begin{array}{ccc}
0 &\delta_3     &\delta_5  \\
0&0&0\\
0     &  \delta_4&0
\end{array}\right)$. 
Verifying the Jordan superidentity for the basis elements $J^{s}(f_2;e_1,e_2,f_2)$, $
J^{s}(f_3;e_1,e_2,f_3)$, $J^{s}(f_3;e_1,e_1,f_2)$, $J^{s}(f_2;e_1,e_1,f_3)$, $J^{s}(f_2;e_1,e_1,f_2)$, we obtain the following system:
\begin{equation} \label{sysJ231}
\begin{alignedat}{2}
\omega_{12}&=0, &\qquad \omega_{13}&=0, \\
\theta _{13} &=0, &\qquad \theta _{12}&=-2 \omega _{23} \delta _4.\\
\end{alignedat}
\end{equation}
Under such conditions, we can verify that the nilpotency of $L_{f_{2}}$ implies: \[\theta_{12}=-2 \omega_{23} \delta _4=0,\] 
since $L_{f_{2}}^{5}(e_{2})=4\omega_{23}^{2}\delta_{4}^{2}f_{1}$, and no further conditions are imposed on the parameters $\delta_{3},\delta_{4},\delta_{5},\omega_{23}$ and $\theta_{23}$. We have the following possibilities:
\begin{enumerate}[a)]
    \item $\omega_{23}=\delta_{4}=0$ and $ \delta_5=0$:
    \begin{enumerate}[i)]
        \item If $\theta_{23}=0$, by replacing $e_{1}^{\prime}\mapsto e_{1}-\delta_{3}e_{2}$, we have $J\cong(2,3)_{35}$.
        \item If $\theta_{23}\neq0$, by changing $e_{1}^{\prime}\mapsto e_{1}-\delta_{3}e_{2}, \ f_{3}^{\prime}\mapsto \theta_{23}^{-1}f_{3}$, we obtain $J\cong (2,3)_{36}$.
    \end{enumerate}
        \item $\omega_{23}\neq0$, $\delta_{4}=0$ and $\delta_{5}=0$:
    \begin{enumerate}[i)]
        \item If $\theta_{23}=-\omega_{23}\delta_{3}$, by changing $e_{1}^{\prime}\mapsto \omega_{23}e_{1}+\theta_{23}e_{2}, \ e_{2}^{\prime}\mapsto \omega_{23}^{2}e_{2}, \ f_{1}^{\prime}\mapsto \omega_{23}^{2}f_{1}$, we obtain $J\cong(2,3)_{37}$.
        \item If $\theta_{23}\neq -\omega_{23}\delta_{3}$, then $k=\theta_{23}+\omega_{23}\delta_{3}\neq0$ and by changing $e_{1}^{\prime}\mapsto k\omega_{23}^{-1}e_{1}+k\theta_{23}\omega_{23}^{-2}e_{2}, \ e_{2}^{\prime}\mapsto k^{2}\omega_{23}^{-2}e_{2}, \ f_{1}^{\prime}\mapsto k^{2}\omega_{23}^{-2}f_{1}, \ f_{3}^{\prime}\mapsto k\omega_{23}^{-2}f_{3}$, we get $J\cong(2,3)_{38}$.
\end{enumerate}
    \item $\omega_{23}=\delta_{4}=0$ and $\delta_5\neq0$:
    \begin{enumerate}[i)]
        \item If $\theta_{23}=0$, by changing $f_{2}^{\prime}\mapsto f_{2}-\delta_{3}\delta_{5}^{-1}f_{3}, \ f_{3}^{\prime}\mapsto\delta_{5}^{-1}f_{3}$, we obtain $J\cong(2,3)_{39}$.
        \item If $\theta_{23}\neq0$, by changing $e_{1}^{\prime}\mapsto e_{1}-\delta_{3}e_{2}, \ f_{i}^{\prime}\mapsto\sqrt{\delta_{5}\theta_{23}^{-1}}f_{i}$, with $i=1,2$, and $f_{3}^{\prime}\mapsto (\sqrt{\delta_{5}\theta_{23}})^{-1}f_{3}$, we get $J\cong(2,3)_{40}$.
    \end{enumerate}
    \item $\omega_{23}\neq0$, $\delta_{4}=0$ and $\delta_{5}\neq0$: In this case, we obtain $J\cong(2,3)_{41}$ via the change of basis $e_{1}^{\prime}\mapsto e_{1}+\theta_{23}\omega_{23}e_{2}, \ f_{1}^{\prime}\mapsto \sqrt{\delta_{5}\omega_{23}^{-1}}f_{1}, \ f_{2}^{\prime}\mapsto\sqrt{\delta_{5}\omega_{23}^{-1}}f_{2}-k\delta_{5}^{-1/2}\omega_{23}^{-3/2}f_{3}, \ f_{3}^{\prime}\mapsto (\sqrt{\omega_{23}\delta_{5}})^{-1}f_{3}$, with $k=\theta_{23}+\omega_{23}\delta_{3}$.
    \item $\omega_{23}=0,\delta_4\neq0$ and $\delta_5=0$:
    \begin{enumerate}[i)]
        \item If $\theta_{23}=0$, by replacing $f_{3}^{\prime}\mapsto \delta_{3}f_{1}+\delta_{4}f_{3}$, we have $J\cong(2,3)_{42}$.
        \item If $\theta_{23}\neq0$, by changing $e_{1}^{\prime}\mapsto \theta_{23}\delta_{4}e_{1}, \ e_{2}^{\prime}\mapsto\theta_{23}^{2}\delta_{4}^{2}e_{2}, \ f_{1}^{\prime}\mapsto \theta_{23}^{2}\delta_{4}^{2}f_{1}, \ f_{3}^{\prime}\mapsto \theta_{23}\delta_{4}(\delta_{3}f_{1}+\delta_{4}f_{3})$, we obtain $J\cong(2,3)_{44}^{0}$.
    \end{enumerate}
    \item $\omega_{23}=0,\delta_4\neq0$ and $\delta_5\neq0$:
    \begin{enumerate}[i)]
        \item If $\theta_{23}=0$, by taking $\gamma=\delta_{4}\delta_{5}$ and replacing $f_{3}^{\prime}\mapsto \delta_{3}f_{1}-\delta_{4}f_{3}$, we have $J\cong(2,3)_{43}^{\gamma}$.
        \item If $\theta_{23}\neq 0$, by taking $\phi=\delta_4\delta_5$ and applying the change of basis $f_{i}^{\prime}\mapsto(\sqrt{\theta_{23}\delta_{4}})^{-1}f_{i}$, with $i=1,2$, and $f_{3}^{\prime}\mapsto \delta_{3}(\sqrt{\theta_{23}\delta_{4}})^{-1}f_{1}+\sqrt{\delta_{4}\theta_{23}^{-1}}f_{3}$, we obtain $J\cong(2,3)_{44}^{\phi}$.
    \end{enumerate}

\end{enumerate}
    \item $[l_{e_{2}}]_{\mathcal{B}_{1}}=\left(\begin{array}{ccc}
0 &1     &0  \\
0     & 0 & 1\\
0&0&0
\end{array}\right)$. By \eqref{operadorescaso23naotrivial}, we have $[l_{e_1}]_{\mathcal{B}_{1}}=\delta_{3} E_{13}$, $\delta_{3}\in \mathbb{C}$. But in this case, $J^{s}(e_2;e_1,e_1,f_3)=f_1$. Therefore, $l_{e_{2}}$ cannot assume this Jordan normal form.
\end{enumerate}

Finally, it remains to verify that $J=(2,3)_{i}^{k_{1}}$ and $J^{\prime}=(2,3)_{i}^{k_{2}}$ are isomorphic if and only if $k_{1}=k_{2}$, for $i\in\{31,43,44\}$. In fact, assume that there exists an isomorphism of superalgebras $\varphi: J\to J^{\prime}$. With respect to the homogeneous
basis $\mathcal{B}=\{e_{1},e_{2},f_{1},f_{2},f_{3}\}$, we can write:
\[[\varphi]_{\mathcal{B}}=\left(\begin{array}{cc}
  T   & 0 \\
  0   & S
\end{array}\right), \, T=(t_{ij})\in \mathrm{GL}_{2}(\mathbb{C}), \, S=(s_{ij})\in \mathrm{GL}_{3}(\mathbb{C}).\]

Note that $T$ represents an automorphism of $J_{2,2}$, then $t_{12}=0$ and $t_{22}=t_{11}^{2}\neq0$. First, consider the case $i=31$. From $S(e_{1}f_{2})=T(e_{1})S(f_{2})$, comparing the coefficients with respect to the basis, we get $s_{21}=s_{31}=0$ and $s_{11}=t_{11}s_{22}\neq0$. From $S(e_{1}f_{3})=T(e_{1})S(f_{3})$, we obtain $s_{23}=0$. Applying these relations recursively and using $T(f_{2}f_{3})=S(f_{2})S(f_{3})$ and  $T(f_{1}f_{3})=S(f_{1})S(f_{3})$ we get $t_{11}=s_{22}s_{33}$ and $k_{1}t_{11}^{2}=s_{11}s_{33}k_{2}$, respectively. Then $t_{11}^{2}k_{1}=t_{11}^{2}k_{2}$, which implies $k_{1}=k_{2}$. For $i=43$ and $i=44$, a similar argument applies, using the conditions $S(e_{1}f_{2})=T(e_{2})S(f_{2})$, $S(e_{1}f_{2})=T(e_{1})S(f_{2})$ and $S(e_{1}f_{3})=T(e_{1})S(f_{3})$. In this way, we can also conclude $k_{1}=k_{2}$. 
\end{proof}

We remark that, by analyzing the invariants presented in Table \ref{table:caso14} and Table \ref{table:caso32}, and considering that the tables are divided according to the even part, it is straightforward to check that the superalgebras listed in them are pairwise non-isomorphic. The same applies to Table \ref{table:caso23}, except for the  following pairs of superalgebras whose last four columns in Table \ref{table:caso23} coincide.
\begin{enumerate}[1)]
    \item $(2,3)_{5}\not\cong(2,3)_{8}$, $(2,3)_{10}\not\cong(2,3)_{18}$, $(2,3)_{12}\not\cong(2,3)_{19}$, $(2,3)_{27}\not\cong (2,3)_{32}$ and $(2,3)_{39}\not\cong(2,3)_{43}^{\gamma}$: 
    it is simple to verify by observing the graded dimensions of the second and third powers of each superalgebra.
\item $(2,3)_{30}\not\cong(2,3)_{31}^{\lambda}$: in this case, we can apply a similar argument used to verify the isomorphisms in the families above, and conclude that there is no $\lambda\in \mathbb{C}^{\ast}$ for which the isomorphism occurs.   
\end{enumerate}

 \section{The geometric classification}\label{classgeo}
 For each type $(m,n)$ of five-dimensional nilpotent Jordan superalgebras determined in the previous section, we now describe all degenerations between  superalgebras in the corresponding variety $\mathcal{NJS}^{(m,n)}$ and justify all the cases of non-degeneration. As a consequence, we determine its irreducible components.

 We remark that the relation $J\rightarrow J^{\prime}$ induces a partial order on $\mathcal{NJS}^{(m,n)}$. To determine the associated Hasse diagram, we first apply item \ref{item:aut}) of Lemma \ref{invariants}. We therefore omit from the tables of non-degenerations the cases that follow immediately from the dimensions of the corresponding automorphism groups. In the tables of degenerations, we list only those that cannot be obtained by transitivity.
 
 In order to prove degenerations, we follow the method employed in \cite{4JorsupI} to determine deformations. Let $J\in \mathcal{NJS}^{(m,n)}$ and consider a family of parametrized matrices $g(t)\in\mathrm{Mat}_{m}(\mathbb{C}(t))\times\mathrm{Mat}_{n}(\mathbb{C}(t))$, where $\mathbb{C}(t)$ denotes the field of fractions of the polynomial ring $\mathbb{C}[t]$, such that $g(t)\in G$ for any $t\in \mathbb{C}^{\ast}$. Let $J_{t}=(c_{ij}^{k}(t),\rho_{ij}^{k}(t),\Gamma_{ij}^{k}(t))$ be the image of $J$ under the action of $g(t)$. If for $t=0$ we obtain $J^{\prime}=(c_{ij}^{k}(0),\rho_{ij}^{k}(0),\Gamma_{ij}^{k}(0))$, then $J\rightarrow J^{\prime}$. For instance, any superalgebra degenerates to the trivial one by taking $g(t)=t^{-1}(I_{m}\times I_{n})$, where $I_{k}$ denotes the identity in $\mathrm{GL}_{k}(\mathbb{C})$. In Tables \ref{table:deg14}, \ref{table:deg32} and \ref{table:degcaso23} , when we prove that $J\rightarrow J^{\prime}$, we describe the parametrized change of basis $\{E_{1}^{t},\dots,E_{m}^{t},F_{1}^{t},\dots,F_{n}^{t}\}$ induced by the action of $g(t)$.

 



Applying Lemma \ref{invariants}, Lemma \ref{newinvariants} and Remark \ref{nondegalg}, we were able to describe almost all non-degenerations, except for one, namely $(2,3)_{30}\nrightarrow (2,3)_{7}$. In this case, we exhibit a closed subset $\mathcal{R}$ of $\mathcal{NJS}^{(2,3)}$ defined by polynomial equations in the structure constants such that $O((2,3)_{30})$ intersects $\mathcal{R}$, whereas $O((2,3)_{7})$ does not.  We conclude the non-degeneration by applying the following result.
\begin{proposition}\cite[Proposition 1.17]{Borel}
Let $G$ be a reductive algebraic group over $\mathbb{C}$ with Borel subgroup $B$, and let $X$ be an algebraic set on which $G$ acts rationally. For $x\in X$, we have $\overline{G\cdot x} =G\cdot\overline{B\cdot x}$.
\end{proposition}
For brevity, in what follows, we denote $(m,n)_i\rightarrow(m,n)_{j}$ and $(m,n)_{k}\nrightarrow(m,n)_{l}$ simply by $i\rightarrow j$ and $k\nrightarrow l$, respectively. Moreover, $k_{1}, \ldots, k_{r} \nrightarrow l_{1}, \ldots, l_{s}$  means that $k_{p} \nrightarrow l_{q}$ for all $1\leq p\leq r$ and $1\leq q\leq s$. In each case, the type $(m,n)$ is indicated in the corresponding table. When we apply item \ref{item:PI}) of Lemma \ref{invariants} to justify a non-degeneration, we specify a polynomial $p(x_{1},\dots,x_{m},y_{1},\dots,y_{n})\in \mathbb{C}\{ X\cup Y\}$, where $X$ and $Y$ denote the sets of even and odd variables, respectively, such that $p$ is a graded identity for $J$, but not for $J^{\prime}$.

\subsection{Varieties $\mathcal{NJS}^{(0,5)}$, $\mathcal{NJS}^{(5,0)}$ and $\mathcal{NJS}^{(4,1)}$}
First, we observe that the variety $\mathcal{NJS}^{(0,5)}$ consists of a single point, so there are no degenerations to consider. On the other hand, $\mathcal{NJS}^{(5,0)}$ corresponds to the variety of five-dimensional nilpotent Jordan algebras, which are geometrically classified in \cite{Geo5Jor}. 

For the variety $\mathcal{NJS}^{(4,1)}$, note that the action of $G$ is completely determined by the action of $\mathrm{GL}_{4}(\mathbb{C})$ on the even part, since all products involving the odd part are zero. In particular, $\mathrm{GL}_{1}(\mathbb{C})$ acts trivially. Therefore,  $J,J^{\prime}\in \mathcal{NJS}^{(4,1)}$ lie in the same orbit if and only if their corresponding even parts lie in the same orbit in the variety  of four-dimensional nilpotent Jordan algebras, $\mathcal{J}\mathrm{orN}_{4}$, which was determined in \cite{Anc}, with some corrections pointed out in \cite{Degnilp}. Consequently, $J\rightarrow J^{\prime}$ if and only if $J_{0}\rightarrow J_{0}^{\prime}$, and the Hasse diagram of $\mathcal{NJS}^{(4,1)}$ coincides with that of $\mathcal{J}\mathrm{orN}_{4}$.
\begin{figure}[H]
\centering
\includegraphics[scale=0.2]{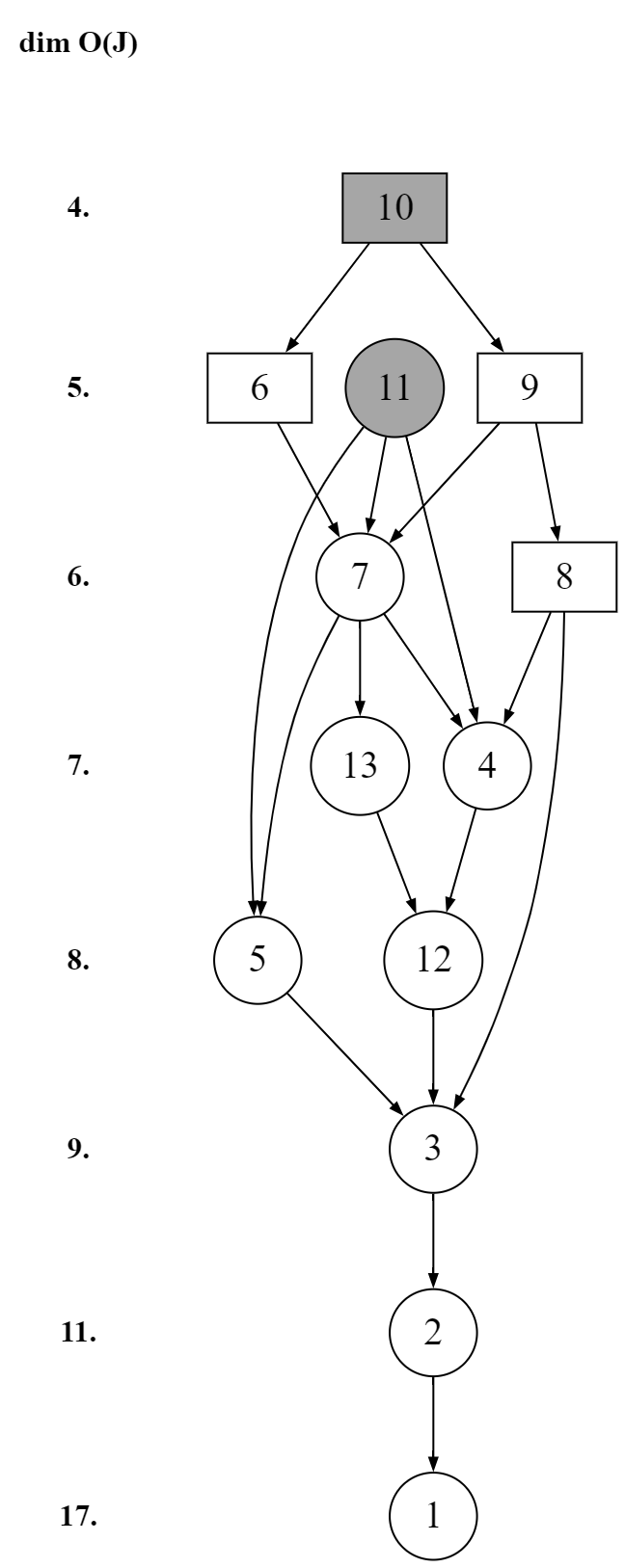}
\caption{Hasse diagram of degenerations for nilpotent Jordan superalgebras of type $(4,1)$.}
\end{figure} 
For the Hasse diagrams, we use the following notation: the gray color indicates a single superalgebra (or a parametrized family of superalgebras) such that the closure of its orbit (or the union of the closures of their orbits) determines an irreducible component of the variety, and a circle represents an associative superalgebra. For abreviation, we write $i$ instead of $(m,n)_{i}$.
\subsection{Variety $\mathcal{NJS}^{(1,4)}$}

\begin{theorem}\label{theo:irrecom14}
The variety $\mathcal{NJS}^{(1,4)}$ has four irreducible components given by 
\[\mathcal{C}_{1}=\overline{O((1,4)_{3})}, \ \mathcal{C}_{2}=\overline{O((1,4)_{6})}, \ \mathcal{C}_{3}=\overline{O((1,4)_{7})}, \ \mathcal{C}_{4}=\overline{O((1,4)_{9})}.\]
In particular, the rigid superalgebras are exactly $(1,4)_{i}$, with $i\in\{3,6,7,9\}$.
\end{theorem}

\begin{proof}
By analyzing Table \ref{table:deg14}, we observe that every $J\in\mathcal{NJS}^{(1,4)}$ belongs to some $\mathcal{C}_{i}$, with $i=1,\dots,4$. Moreover, by Table \ref{table:ndegcaso14} and the dimensions of the corresponding automorphism groups, there are no degenerations between the superalgebras $(1,4)_{3}$, $(1,4)_{6}$, $(1,4)_{7}$ and $(1,4)_{9}$. Hence, $\mathcal{C}_{i}\not\subseteq\mathcal{C}_{j}$, whenever $i\neq j$. Since the closure of a single orbit is irreducible, we conclude that $\mathcal{C}_{1}$, $\mathcal{C}_{2}$,$\mathcal{C}_{3}$ and $\mathcal{C}_{4}$ are precisely the irreducible components of $\mathcal{NJS}^{(1,4)}$. Therefore, the superalgebras $(1,4)_{i}$, with $i\in\{3,6,7,9\}$, are rigid in the variety.
\end{proof}

{\footnotesize
\begin{longtable}{|l|l|l|l|}
 \caption{Degenerations between nilpotent Jordan superalgebras of type $(1,4)$.}  \label{table:deg14}\\ \hline
 $ J \rightarrow J^{\prime} $ & {\text{Change of Basis}}& $ J \rightarrow J^{\prime} $&{\text{Change of Basis}}\\ \hline
 $ 3 \rightarrow 2 $ & $\{e_{1},\,f_1,\,f_2,\,tf_3,\,f_4 \} $  &$8 \rightarrow 4 $ & $\{e_{1},\,f_1,\,f_2,\,t^{-1}f_3,\,f_4\}$ \\ \hline
 $ 5 \rightarrow 2 $ & $\{e_{1},\,tf_2,\,t^{-1}f_3,\,f_1,\,f_4 \}$ &$ 9 \rightarrow 2 $ & $ \{te_{1},\,-f_4,\,tf_2,\,f_3,\,tf_1\,\}$ \\ \hline
 $ 5 \rightarrow 4 $ & $\{e_{1},\,f_1,\,f_2,\,tf_3,\,f_4\}$ &$9 \rightarrow 4$ & $\{e_{1},\,f_1,\,f_2,\,tf_3,\,tf_4\} $ \\ \hline
$ 6 \rightarrow 5 $ & $\{e_{1},\,f_1,\,f_2-f_4,\,f_3,\,tf_4\}$ &$9 \rightarrow8$ & $\{e_{1},\,f_1,\,f_2,\,tf_3,\,tf_4\} $ \\ \hline
$ 7 \rightarrow 8 $ & $\{e_{1},\,f_1,\,f_2+f_4,\,tf_2,\,tf_3\}$ \\ \cline{1-2}
\end{longtable}
}

{\footnotesize
\begin{longtable}{|l|l|}
\caption{Non-degenerations between nilpotent Jordan superalgebras of type $(1,4)$.}\label{table:ndegcaso14}\\\hline
 $ J\nrightarrow J^{\prime} $  & \text{Reason} \\ \hline
 $ 4 \nrightarrow 2;\,\,\, 5,6,9 \nrightarrow 3;\,\,\, 7 \nrightarrow 2,3,5,6, 9;$  &\multirow{2}{*}{ $ a(J) \nrightarrow a(J^{\prime}) $} \\ 
$8\nrightarrow2,3.$&\\ \hline
$ 6 \nrightarrow 8. $ & $ F(J) \nrightarrow F(J^{\prime}) $ \\ \hline

$ 9 \nrightarrow 5. $ & $ \mathrm{dim}(\mathrm{Ann}(J))_{1} > \mathrm{dim}(\mathrm{Ann}(J^{\prime}))_{1} $ \\ \hline
\end{longtable}
}

Figure \ref{fig:Hasse23} summarizes the information contained in Table \ref{table:deg14} and Table \ref{table:ndegcaso14}, and gives a picture of the facts presented in Theorem \ref{theo:irrecom14}.

\begin{figure}[H]
\centering
\includegraphics[scale=0.2]{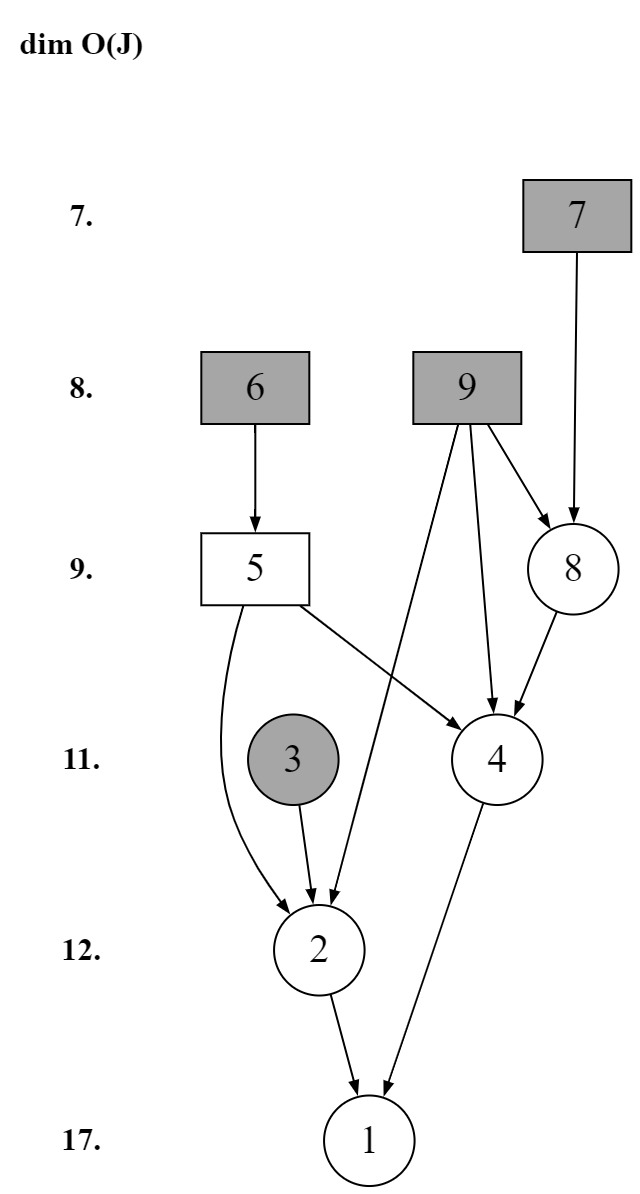}
\caption{Hasse diagram of degenerations for nilpotent Jordan superalgebras of type $(1,4)$.}
\end{figure} 
\subsection{Variety $\mathcal{NJS}^{(3,2)}$}
Using Tables \ref{table:deg32} and \ref{table:ndeg32}, we can argue as in the proof of Theorem \ref{theo:irrecom14} to obtain the following result, as can be seen in Figure \ref{fig:hasse32}.

\begin{theorem}
The variety $\mathcal{NJS}^{(3,2)}$ has six irreducible components given by
\[\mathcal{C}_{1}=\overline{O((3,2)_{18})}, \ \mathcal{C}_{2}=\overline{O((3,2)_{19})}, \ \mathcal{C}_{3}=\overline{O((3,2)_{23})}, \]\[ \mathcal{C}_{4}=\overline{O((3,2)_{25})}, \mathcal{C}_{5}=\overline{O((3,2)_{27})}, \ \mathcal{C}_{6}=\overline{O((3,2)_{29})}.\]
In particular, the rigid superalgebras are exactly $(3,2)_{i}$, with $i\in\{18,19,23,25,27,29\}$.    
\end{theorem}
{\footnotesize
\begin{longtable}{|l|l|l|l|}
\caption{Degenerations between nilpotent Jordan superalgebras of type $(3,2)$.}\label{table:deg32}\\ \hline
$ J \rightarrow J^{\prime} $ & \text{Change of Basis}&$ J \rightarrow J^{\prime} $&\text{Change of Basis}\\ \hline
$ 4 \rightarrow 2$ & $\{te_1,\,e_2,\,e_3,\,f_1,\,f_2\} $ &$ 20 \rightarrow 12 $ & $\{e_1+e_2,\,2e_3,\,te_2,\,f_1,\,f_2\} $ \\ \hline
$ 4 \rightarrow 3 $ & $\{e_1,\,e_2,\,e_3,\,tf_1,\,tf_2\} $ &$ 20 \rightarrow 15 $ & $\{e_1,\,e_2,\,e_3,\,f_1,\,tf_2\}$ \\ \hline

$6 \rightarrow 8$ & $\{e_1+e_3,\,e_2,\,te_1,\,f_1,\,tf_2\}$ &$ 21 \rightarrow 13 $ & $\{e_1+{\scriptstyle\frac{1}{2}}e_2,\,e_3, \,te_2,\,f_1,\,f_2\}$ \\ \hline
$ 7\rightarrow 2 $ & $\{te_1,\,e_2,\,e_3,\,f_1,\,f_2 \} $ &$ 21 \rightarrow 17 $ & $\{e_1,\,te_2,\,te_3,\,f_1,\,tf_2 \}$ \\ \hline
$ 7 \rightarrow 5 $ & $\{e_1,\,e_2,\,e_3,\,tf_1,\,f_2\}$ &$ 21 \rightarrow 20 $ & $\{e_1,\,e_2,\,e_3,\, tf_1,\,tf_2\}$ \\ \hline

$ 8 \rightarrow 7 $ & $\{e_1,\,e_2,\,e_3-t^{-1}e_2,\,tf_1,\,f_2 \}$ &$ 22 \rightarrow 10 $ &$\{e_1-e_2,\,-2e_3,\,-2te_2,\,-2tf_1,\,f_2\}$ \\ \hline
$ 9 \rightarrow 10 $ &$\{e_1,\,e_2,\,t^{-1}e_2+e_3,\,f_1,\,tf_2\}$ &$ 22 \rightarrow 20 $ & $\{e_1,\,te_2,\,te_3,\,f_1,\,f_2\}$ \\ \hline
$ 10 \rightarrow 12 $ & $\{e_1+e_3,\,e_2,\,te_3,\,f_1,\,f_2\} $ &$ 23 \rightarrow 11 $ & $\{t(e_1-e_2),\,-2t^{2}e_3,\,-2t^{2}e_2, \,2t^{2}f_1,\,-f_2 \}$ \\ \hline
$ 11 \rightarrow 10 $ & $\{e_1,\,e_2,\,e_3,\,tf_1,\,tf_2\}$ &$ 23 \rightarrow 21 $ & $\{te_2,\,t^{2}e_1,\,t^{3}e_3,\,t^{2}f_1,\,tf_2 \}$ \\ \hline
$ 11 \rightarrow 13$ & $\{te_1+e_3,\,t^{2}e_2,\,te_3,\,tf_1,\,tf_2 \}$ &$ 23 \rightarrow 22 $ & $\{e_1,\,e_2, \,e_3,\,tf_1,\,tf_2\}$ \\ \hline
$ 12 \rightarrow 3 $ & $\{te_1,\,e_2,\,e_3,\,tf_1,\,f_2\}$ &$ 24 \rightarrow 15 $ & $\{te_1,\,e_2,\,te_3,\,f_1,\,f_2\}$ \\ \hline
$12 \rightarrow 5 $ & $\{e_1,\,e_2,\,e_3,\,f_1,\,tf_2\}$ &$ 25 \rightarrow 16 $ & $\{t(e_1+e_2),\,t(e_2+2e_3),\,t^{2}e_3,\,tf_1,\,f_2\}$ \\ \hline
$ 13 \rightarrow 4 $ & $\{te_1,\,te_2,\,e_3,\,tf_1,\,f_2\}$ &$ 25 \rightarrow 26$ & $\{te_1-{\scriptstyle\frac{1}{2}}te_2,\,t^{2}(e_2-e_3),\,t^{3}e_3,\,f_1,\,t^{3}f_2\}$ \\ \hline
$ 13 \rightarrow 7 $ & $\{te_1,\,t^{2}e_2,\,e_3,\,tf_1,\,tf_2\} $ &$ 26 \rightarrow 8 $ & $\{te_1,\,t^{2}e_2,\,e_3,\,f_1,\,f_2\}$ \\ \hline 
$13\rightarrow12$&$\{e_{1},\, e_{2},\, e_{3},\, f_{1},\, tf_{2}\}$&$ 26 \rightarrow 17 $ & $\{te_1,\,t^{2}e_2,\,t^{3}e_3,\,tf_1,\,t^{2}f_2 \}$ \\ \hline
$ 14 \rightarrow 8 $ & $\{te_1,\,t^{2}e_2,\,e_3,\,f_1,\,f_2 \}$ &$ 26 \rightarrow 24 $ & $\{e_1,\,e_2,\,e_3,\,tf_1,\,f_2\}$ \\ \hline
$ 14 \rightarrow 13 $ & $\{e_1,\,e_2, \,e_3-t^{-1}e_2,\,f_1, \,tf_2\}$ &$ 27 \rightarrow 9 $ & $\{te_1,\,t^{2}e_2,\,e_3,\,t^{2}f_1,\,f_2\}$ \\ \hline
$ 15 \rightarrow 5 $ & $\{e_1+e_2,\,2e_3,\,2te_2,\,f_1,\,f_2\}$ &$ 27 \rightarrow 22 $ & $\{t(e_1+e_2),\,t(e_2+2e_3),\,t^{2}e_3,\,tf_1,\,f_2\}$ \\ \hline
$ 16 \rightarrow 8 $ & $\{e_1+e_2,\,2e_3,\,-2te_2,\,-2tf_1,\,f_2\}$ &$ 27 \rightarrow 28 $ & $\{t(e_1+e_2),\,t^{2}(e_2+2e_3),\,t^{3}e_3,\,tf_1,\,f_2 \}$ \\ \hline
$ 16 \rightarrow 17 $ & $\{te_1-e_3,\,e_2,\,te_3,\,f_1,\,t^{2}f_2\}$ &$ 28 \rightarrow 20 $ & $\{te_1,\,e_2,\,te_3,\,tf_1,\,f_2\}$ \\ \hline

$ 17 \rightarrow 7 $ & $\{e_1+e_2,\,2e_3,\,2te_2,\,2f_1,\,f_2\}$ & $ 28 \rightarrow 24 $ & $\{e_1,\,e_2,\,e_3,\,f_1,\,tf_2 \}$ \\ \hline

$ 17\rightarrow 15 $ & $\{e_1,\,e_2,\,e_3,\,tf_1,\,f_2\}$ &$ 29 \rightarrow 14$ & $\{te_1,\,t^{2}e_2,\,te_3,\,tf_1,\,f_2 \}$ \\ \hline
$ 18 \rightarrow 6 $ & $\{e_1+e_2,\,2e_3,\,2te_2,\,f_1,\,f_2 \}$ &$ 29 \rightarrow 21 $ & $\{te_1,\,e_2,\,te_3,\,tf_1,\,f_2 \}$ \\ \hline
$ 18 \rightarrow 16 $ & $\{te_1,\,e_2,\,te_3,\,f_1,\,tf_2\}$ &$ 29\rightarrow 26 $ & $\{e_1,\,e_2,\,e_3,\,t^{-1}f_1,\,tf_2\} $ \\ \hline
$ 19 \rightarrow 9 $ & $\{e_1+e_2,\,2e_3,\,2te_2,\,2f_1,\,f_2\}$&$ 29 \rightarrow 28 $ & $\{e_1,\,e_2,\,e_3, \,tf_1,\,tf_2\}$ \\ \hline
$ 19 \rightarrow 22 $ & $\{e_2+e_3,\,te_1+e_3,\,te_3,\,f_1,\,f_2\}$\\ \cline{1-2}
\end{longtable}
}
{\footnotesize
 \begin{longtable}{|l|l|}
 \caption{Non-degenerations between nilpotent Jordan superalgebras of type $(3,2)$.}\label{table:ndeg32}\\ \hline
$ J \nrightarrow J^{\prime} $ & $ \text{Reason} $ \\ \hline
$ 3,5 \nrightarrow 2; \,\,\,\, 9 \nrightarrow 2,4,6,7,8,13;$ & \multirow{6}{*}{$ a(J) \nrightarrow a(J^{\prime}) $} \\
$10\nrightarrow2,4,7,8; \,\,\, \,12,24\nrightarrow2,7;$&\\
$19 \nrightarrow 2,4,6,7,8,11,13,14,16,17,18,21; $ & \\ 
$20\nrightarrow2,4,7,8,17; \,\,\,\, 22\nrightarrow 2,4,6,7,8,13,16,17;$&\\
$ 27 \nrightarrow 2,4,6,7,8,11,13,14,16,17,18,21,25,26; $ & \\ 
$ 28 \nrightarrow 2,4,6,7,8,13,16,17,26. $ & \\ \hline
$4 \nrightarrow 5,7,15;\,\,\, \, 6,10,13 \nrightarrow  15,17,24; $ & \multirow{4}{*}{$ J_{0} \nrightarrow J_{0}^{\prime} $} \\ 
$ 8,12 \nrightarrow15; \,\,\, \,9,11,14 \nrightarrow 15,16,17,20; $&\\
$16,20\nrightarrow 24;\,\,\, \,  18,21,22\nrightarrow24,26;$&\\
$19,23 \nrightarrow 24,25,26,28.$ &  
\\ \hline
$6,16,26 \nrightarrow 3,4,12;\,\,\,\, 7,8,17 \nrightarrow 3; \,\,\,\, 14\nrightarrow10;$ & \multirow{3}{*}{$F(J) \nrightarrow F(J^{\prime}) $}  \\ 
$18,25\nrightarrow 3,4,10,12,13,20; \,\,\,\, 21\nrightarrow10;$&\\
$23 \nrightarrow 9;\,\,\,\, 29 \nrightarrow 9,10,20,22. $ & \\ 
 \hline
 $11 \nrightarrow 6,8;\,\,\,\, 13\nrightarrow8;\,\,\,\, 14,25\nrightarrow 6; $ & \multirow{2}{*}{$\mathrm{dim}(J^{2})_{0} < \mathrm{dim}({J^{\prime}}^{2})_{0} $} \\
 $21\nrightarrow 6,8,16;\,\,\,\, 23 \nrightarrow 6,8,14,16,18.$&\\\hline
$ 15 \nrightarrow 2, 3; \,\,\,\, 24 \nrightarrow 3.$ &$ \mathrm{dim}(\mathrm{Ann}(J))_{1} > \mathrm{dim}(\mathrm{Ann}(J^{\prime}))_{1} $  \\  \hline

$ 29 \nrightarrow 6,18.$ & $ \mathrm{nilindex}(J) < \mathrm{nilindex}(J^{\prime}) $ \\ \hline

$ 29 \nrightarrow 16,25.$ & $ p(x_1,y_1,y_2)=(y_{1}y_{2})x_{1} $ \\ \hline
\multirow{3}{*}{$28 \nrightarrow 10. $ }& $ J \nrightarrow J^{\prime}$ as algebras (see \cite{Degnilcomasso5}, where  \\ &
$28$ and $10$ corresponds to $A_{10}$ and \\
& $A_{15}$, respectively.) \\\hline

\end{longtable} 
}

\begin{figure}[H]
\centering
\includegraphics[scale=0.2]{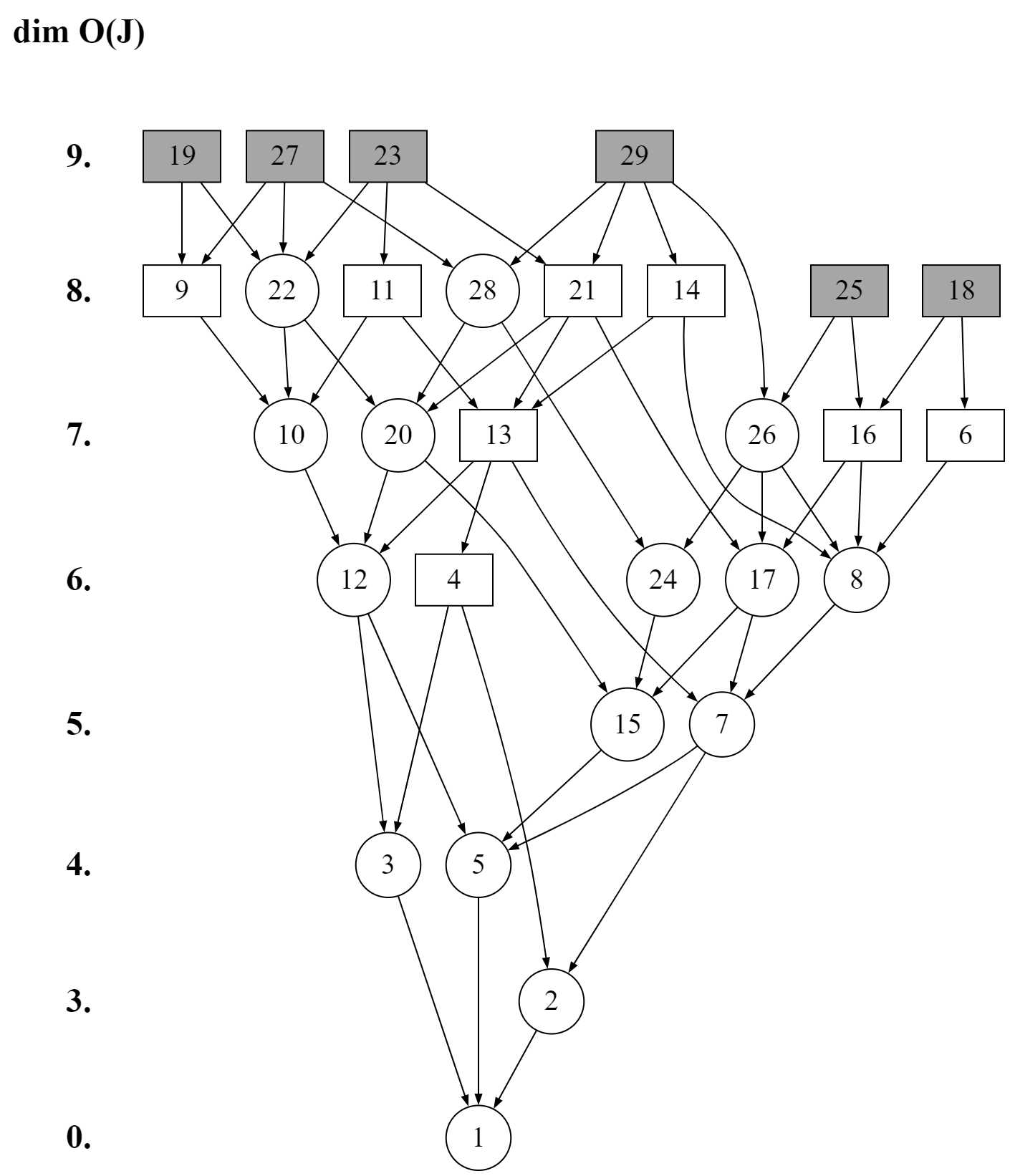}
\caption{Hasse diagram of degenerations for nilpotent Jordan superalgebras of type $(3,2)$.}
\label{fig:hasse32}
\end{figure}
\subsection{Variety $\mathcal{NJS}^{(2,3)}$}

When there are only finitely many orbits, as in the previous cases, the Hasse diagram provides all the necessary information about the variety under consideration. In particular, the rigid superalgebras and the irreducible components can be easily identified. This is no longer the case for $\mathcal{NJS}^{(2,3)}$, since it contains families of infinitely many non-isomorphic superalgebras. In order to determine the irreducible components of this variety, it is also necessary to consider the notion of degeneration involving families of superalgebras, according to \cite{Geo5Jor}. First, to stablish the notation, given a family of superalgebras $(2,3)_{r}^{\alpha}$, where $\alpha\in I\subseteq\mathbb{C}$, we denote by ${\mathcal{N}}_{r}^{\#}=\bigcup\nolimits_{\alpha \in I}O((2,3)_{r}^{\alpha})$ the union of the orbits of superalgebras in the family $(2,3)_{r}^{\alpha}$.

Let $(2,3)_{j}$ and $(2,3)_{l}$ be single superalgebras. We write $(2,3)_{j}\rightarrow (2,3)_{r}^{\alpha}$ if $(2,3)_{r}^{\alpha}\in\overline{O((2,3)_{j})}$ for every $\alpha\in I$. On the other hand, we write $(2,3)_{r}^{\alpha}\rightarrow (2,3)_{l}$ if $(2,3)_{l}\in \overline{{\mathcal{N}}_{r}^{\#}}$. More generally, if $(2,3)_{s}^{\beta}$, $\beta\in K\subseteq \mathbb{C}$, is another family of superalgebras, then $(2,3)_{r}^{\alpha}\rightarrow (2,3)_{s}^{\beta}$ if $(2,3)_{s}^{\beta}\in \overline{{\mathcal{N}}_{r}^{\#}} $ for every $\beta\in K$. To prove such degenerations, it is enough to exhibit, for any fixed $\beta_{0}\in K$ (or $(2,3)_{l}$), a pair $(g(t),\alpha(t))$, with $g(t)$ defined as above and $\alpha(t)\in I$ such that the change of basis of  $(2,3)_{r}^{\alpha(t)}$ induced by $g(t)$ describes a curve $\omega(t)$ in $\mathcal{NJS}^{(2,3)}$ that lies transversely to the orbits of $(2,3)_{r}^{\alpha}$, in the sense that $\omega(t)\subset {\mathcal{N}}_{r}^{\#}$ when $t\neq0$, and intersects the orbit of $(2,3)_{s}^{\beta_{0}}$ (or the orbit of $(2,3)_{l}$) at $t=0$. 

For example, let us show that $(2,3)_{31}^{\lambda}\rightarrow (2,3)_{10}$. Let $\omega(t):\{e_1,\,t^{-1}e_2,\,f_1,\,f_2,\,f_3\}$ be a change of basis of $(2,3)_{31}^{\lambda(t)}$, with $\lambda(t)=t^{-1}$. So $\omega$ defines a curve in $\mathcal{NJS}^{(2,3)}$ such that  $\omega(t)\in O((2,3))_{31}^{\lambda(t)}$ for any $t\neq0$, that is $\omega(t)\subset {\mathcal{N}}_{31}^{\#}$, and intersects $O((2,3)_{10})$ at $t=0$. Now, we verify that $(2,3)_{44}^{\gamma_{0}}\rightarrow(2,3)_{43}^{\gamma}$. Indeed, for any fixed $\gamma_{0}\in\mathbb{C}^{\ast} \ \backslash \ \{1\}$, consider the change of basis $\omega(t)=\{e_{1},\,e_{2},\,tf_{1},\,tf_{2},\,tf_{3}\}$ of $(2,3)_{44}^{\phi(t)}$, where $\phi(t)=\gamma_{0}$. In this case, $\alpha(t)\in O((2,3)_{44}^{\gamma_{0}})\subset {\mathcal{N}}_{44}^{\#}$ when $t\neq0$ and intersect $O((2,3)_{43}^{\gamma_{0}})$ at $t=0$, consequently $(2,3)_{43}^{\gamma_{0}}\in \overline{{\mathcal{N}}_{44}^{\#}}$. 


\begin{theorem}
The variety $\mathcal{NJS}^{(2,3)}$ has five irreducible components, two of them are given by  $\mathcal{C}_{1}=\overline{\bigcup\nolimits_{\lambda} \, O((2,3)_{31}^{\lambda})}$, $\lambda\in \mathbb{C}^{\ast}$, and $\mathcal{C}_{2}=\overline{\bigcup\nolimits_{\phi} \, O((2,3)_{44}^{\phi})}$, $\phi \in \mathbb{C}$. The others are given by 
\[\mathcal{C}_{3}=\overline{O((2,3)_{15})}, \ \mathcal{C}_{4}=\overline{O((2,3)_{33})}, \ \mathcal{C}_{5}=\overline{O((2,3)_{41})}.\]
In particular, the rigid superalgebras are exactly $(2,3)_{i}$, with $i\in\{15,33,41\}$.  
\end{theorem}

\begin{proof}
By analyzing Table \ref{table:degcaso23}, we conclude that every $J\in\mathcal{NJS}^{(2,3)}$ belongs to $\mathcal{C}_{i}$ for some $i=1,\dots,5$. This fact can also be easily seen in Figure \ref{fig:Hasse23}. Now, we need to verify that each $\mathcal{C}_{i}$ corresponds to an irreducible component of $\mathcal{NJS}^{(2,3)}$. 

It is immediate that $\mathcal{C}_{3}$, $\mathcal{C}_{4}$ and $\mathcal{C}_{5}$ are irreducible. Note that $\mathcal{C}_{1}$ is the product of two irreducible varieties, $\mathcal{O}((2,3)_{31}^{\lambda}$ and $\mathbb{C}^{\ast}$, as $\mathcal{C}_{2}$ is the product of $\mathcal{O}((2,3)_{44}^{\phi}$ and $\mathbb{C}$. Therefore $\mathcal{C}_{1}$ and $\mathcal{C}_{2}$ are also irreducible. From Table \ref{table:ndeg23}, we conclude that $\mathcal{C}_{i}\not\subset \mathcal{C}_{j}$, whenever $i\neq j$. Therefore, $\mathcal{C}_{1},\dots,\mathcal{C}_{5}$ are the irreducible components of $\mathcal{NJS}^{(2,3)}$ and the rigid superalgebras are precisely $(2,3)_{15}$, $(2,3)_{33}$ and $(2,3)_{41}$.
\end{proof}
{\footnotesize

\begin{longtable}{|l|l|l|l|}
 \caption{Degenerations between nilpotent Jordan superalgebras of type $(2,3)$.}\label{table:degcaso23}\\
\hline   $J\rightarrow J^{\prime}$   & {Change of Basis}& $J\rightarrow J^{\prime}$   & {Change of Basis}\\ \hline
    ${3}\rightarrow{2}$&$\{e_1,\,e_2,\,f_1,\,f_2,\,tf_3\}$&  ${29}\rightarrow{8}$&$\{te_1,\,e_2,\,tf_1,\,f_2,\,tf_3\}$ \\ \hline
    
    ${5}\rightarrow{6}$&$\{e_1,\,e_2,\,tf_1,\,tf_2,\,\,f_3-t^{-1}f_1\}$& ${29}\rightarrow{22}$ &$\{te_1,\,t^{2}e_2,\,-tf_3,\,f_2,\,f_1\}$ \\ \hline
     ${6}\rightarrow{2}$&$\{-e_2,\,te_1,\,f_3,\,f_2,\,f_1\}$& ${29}\rightarrow{26}$& $\{e_1-t^{-1}e_2,\,e_2,\,f_1,\,f_2,\,tf_3\}$ \\ \hline
${6}\rightarrow{4}$&$\{e_1,\,e_2,\,f_1,\,f_2,\,tf_3\}$& ${30}\rightarrow{9}$ & $\{te_1,\,te_2,\,tf_1,\,f_2,\,tf_3\}$  \\ \hline
${7}\rightarrow{5}$& $\{e_1,\,e_2,\,f_1,\,f_2+f_3,\,tf_3\}$& ${30}\rightarrow{24}$& $\{-t^{-1}e_{1},\,t^{-2}e_{2},\,f_{2},\,-t^{-1}f_{3},\,-t^{-2}f_{1}\}$ \\ \hline
 ${8}\rightarrow{6}$&$\{e_1-t^{-1}e_2,\,e_2,\,f_1,\,f_2,\,tf_3\}$ & ${30}\rightarrow{27}$&$\{e_1,\,e_2,\,f_1,\,f_2,\,tf_3\}$\\ \hline
  ${9}\rightarrow{3}$& $\{te_1,\,-e_2,\,f_2,\,tf_3,\,f_1\}$ & ${30}\rightarrow{29}$&$\{e_1,\,e_2,\,tf_1,\,tf_2,\,t^{-1}f_3\}$ \\ \hline
   ${9}\rightarrow{5}$& $\{e_1,\,e_2,\,f_1,\,f_2,\,tf_3\}$ & ${31}^{\lambda}\rightarrow{10}$
&$\{e_1,\,t^{-1}e_2,\,f_1,\,f_2,\,f_3\},\, \lambda(t)=t^{-1}$\\ \hline
   ${9}\rightarrow{8}$& $\{e_1,\,t^{-1}e_2,\,f_1,\,f_2,\,f_3\}$ &${31}^{\lambda}\rightarrow{28}$&$\{e_1,\,e_2,\,tf_1,\,tf_2,\,f_3\}, \, \lambda(t)=t^{-1}$\\ \hline
   ${10}\rightarrow{7}$& $\{e_1,\,te_2,\,f_1,\,f_2,\,tf_3\}$&${31}^{\lambda}\rightarrow{30}$&$\{e_1,\,e_2,\,t^{-1}f_1,\,t^{-1}f_2+f_3,\,tf_3\},\, \lambda(t)=t$ \\ \hline
   ${10}\rightarrow{9}$&$\{te_1,\,te_2,\,tf_1,\,f_2+f_3,\,tf_3\}$ & ${32}\rightarrow{16}$&$\{te_1,\,e_2,\,t^{2}f_1,\,tf_2,f_3\}$  \\ \hline
   ${11}\rightarrow{4}$& $\{e_1,\,e_2,\,f_1,\,f_2,\,tf_3\}$ & ${32}\rightarrow{25}$&$\{e_1,\,e_2,\,f_1,\,f_2,\,tf_3\}$\\ \hline
   ${12}\rightarrow{8}$&$\{te_2,\,t^{2}e_1,\,tf_1,f_1+f_3,\,-tf_2\}$  & ${33}\rightarrow{18}$ & $\{te_1,\,e_2,\,tf_1,\,f_2,\,t^{-1}f_3\}$ \\ \hline
 ${12}\rightarrow{11}$&$\{e_1,\,e_2,\,tf_1,\,tf_2,\,tf_3\}$&${33}\rightarrow{28}$& $\{te_1,\,t^{2}e_2,\,tf_1,\,f_2,\,tf_3\}$ \\ \hline
 ${13}\rightarrow{4}$& $\{e_1,\,te_2,\,f_1,\,f_2,\,f_3\}$&${33}\rightarrow{34}$&$\{e_1,\,e_2,\,tf_1,\,tf_2+t^{-1}f_1,\,tf_3+t^{-1}f_2\}$ \\ \hline
 ${14}\rightarrow{5}$& $\{e_1,\,t^{2}e_2,\,tf_1,\,tf_2,\,f_3\}$ &${34}\rightarrow{17}$&$\{te_1,\,te_2,\,t^{2}f_1,\,tf_2,\,f_3\}$ \\ \hline
  ${14}\rightarrow{8}$&$\{-e_2,\,te_1,\,-f_3,\,f_2,\,f_1\}$& ${34}\rightarrow{27}$&$\{e_1,\,e_2,\,f_2,\,f_3,\,t^{-1}f_1\}$ \\ \hline
  ${14}\rightarrow{13}$&$\{e_1,\,e_2,\,tf_1,\,tf_2,\,tf_3\}$ &
   ${34}\rightarrow{32}$&$\{e_1,\,e_2,\,tf_1,\,tf_2,\,tf_3\}$ \\ \hline
  ${15}\rightarrow{19}$& $\{e_1+e_2,\,te_2,\,f_1,\,-t^{-1}f_1+f_2,\,-t^{-1}f_2+f_3\}$  &${35}\rightarrow{25}$&$\{e_1+t^{-1}e_2,\,e_2,f_1,\,tf_2,\,f_3\}$\\ \hline
  ${16}\rightarrow{4}$& $\{e_1,\,e_2,\,f_1,\,f_2,\,tf_3\}$&${36}\rightarrow{8}$&$\{e_2,\,te_1,\,f_1,\,f_2,\,f_3\}$ \\ \hline
      ${17}\rightarrow{5}$& $\{e_1,\,t^{2}e_2,\,tf_2,\,tf_3,\,f_1\}$&${36}\rightarrow{26}$& $\{e_1+t^{-1}e_2,\,e_2,\,f_1,\,tf_2,\,t^{-1}f_3\}$\\ \hline
     ${17}\rightarrow{16}$&$\{e_1,\,e_2,\,tf_1,tf_2,\,tf_3\}$& ${36}\rightarrow{35}$&$\{e_1,\,e_2,\,f_1,\,f_2,\,tf_3\}$ \\ \hline
      ${18}\rightarrow{7}$ & $\{te_1,\,te_2,\,tf_1,\,f_2,\,f_3\}$ & ${37}\rightarrow{22}$&$\{e_1,\,e_2,\,-t^{-1}f_3,\,tf_2,\,f_1\}$ \\ \hline 
 ${18}\rightarrow{17}$&$\{e_1,\,te_2,\,tf_1,\,f_1+tf_2,\,(1+t^{2})f_2+tf_3\}$ &${37}\rightarrow{26}$ & $\{e_1-t^{-1}e_2,\,e_2,\,-f_1,\,tf_2,\,f_3\}$ \\ \hline  

 ${19}\rightarrow{11}$&$\{e_1,\,t^{-1}e_2,\,tf_1,\,tf_2,\,t^{2}f_3\}$&${37}\rightarrow{35}$&$\{e_1,\,e_2,\,f_1,\,f_2,\,tf_3\}$ \\ \hline
 ${19}\rightarrow{13}$& $\{te_2,\,te_1,\,tf_1,\,f_3,\,tf_2\}$& ${38}\rightarrow{29}$& $\{te_1,\,t^{2}e_2,\,tf_1,\,f_2,\,tf_3\}$ \\ \hline
   ${19}\rightarrow{16}$&$\{e_1,\,te_2,\,f_1,\,f_2,\,f_3\}$&${38}\rightarrow{36}$& $\{(2t)^{-1}(e_1-\tfrac{1}{2}e_2),\,(2t)^{-2}e_2,\,(2t)^{-2}f_1,\,f_2,\,(2t)^{-1}f_3\}$ \\ \hline
    ${20}\rightarrow{12}$&$\{t^{2}e_1,\,t^{3}e_2,\,t^{4}f_1,\,t^{2}f_2,\,tf_3\}$ & ${38}\rightarrow{37}$ & $\{t^{-1}e_1,\,t^{-2}e_2,\,t^{-2}f_1,\,f_2,\,t^{-1}f_3\}$ \\ \hline
     ${20}\rightarrow{14}$& $\{te_1,\,te_2,\,tf_2,\,f_3,\,tf_1\}$ & ${39}\rightarrow{11}$&$\{te_1,\,te_2,\,tf_1,\,f_3,\,f_2\}$ \\ \hline
    ${20}\rightarrow{17}$&$\{e_1,\,t^{2}e_2,\,tf_1,\,tf_2,\,tf_3\}$ &${39}\rightarrow{35}$&$\{e_1,\,e_2,\,f_1,\,f_2,\,tf_3\}$ \\ \hline
   
     ${20}\rightarrow{19}$&$\{e_1,\,e_2,\,tf_1,\,tf_2,\,tf_3\}$& ${40}\rightarrow{12}$&$\{-t^{2}e_1,\,-te_{2},\,-t^{2}f_1,\,f_3,\,tf_2\}$\\  \hline
 ${22}\rightarrow{23}$&$\{e_1-t^{-1}e_2,\,e_2,\,tf_1,\,f_2,\,f_3\}$ &${40}\rightarrow{36}$& $\{e_1,\,e_2,\,t^{-1}f_1,\,t^{-1}f_2,\,tf_3\}$ \\ \hline
 ${23}\rightarrow{2}$& $\{e_2,\,te_1,\,f_1,\,f_2,\,f_3\}$&${40}\rightarrow{39}$&$\{e_1,\,e_2,\,tf_1,\,tf_2,\,tf_3\}$ \\ \hline
${23}\rightarrow{21}$& $\{e_1,\,e_2,\,tf_1,\,f_2,\,f_3\}$ & ${41}\rightarrow{38}$&$\{e_1, \,e_2,\,t^{-1}f_{1},\,t^{-1}(f_{2}+f_{3}),\,tf_{3}\}$\\ \hline
${24}\rightarrow{3}$&$\{te_{1},\,e_{2},\,f_{1},\,tf_{2},\,f_{3}\}$ &${41}\rightarrow{40}$& $\{t^{1/2}e_1-e_2,\,te_2,\,t^{3/2}f_1,\,t^{1/2}f_2+f_3,\,tf_3 \}$ \\ \hline 
${24}\rightarrow{22}$ & $\{e_{1},\,e_{2},\,f_{1},\,f_{2},\,tf_{3}\}$& ${42}\rightarrow{13}$&$\{e_2,\,te_1,\,f_1,\,f_2,\,tf_3\}$\\ \hline
${25}\rightarrow{4}$&$\{te_1,\,e_2,\,tf_1,\,f_2,\,f_3\}$& ${42}\rightarrow{35}$&$\{e_1,\,e_2,\,f_1,\,f_2,\,t^{-1}f_3\}$ \\ \hline 
${25}\rightarrow{21}$&$\{e_1,\,e_2,\,f_1,\,tf_2,\,f_3\}$ &${43}^{\gamma}\rightarrow{19}$&$\{e_{1},\,t^{-1}e_{2},\,t^{-1}f_{1},\,f_{3},\,f_{2}\},\, \gamma(t)=t^{-1}$ \\\hline
 ${26}\rightarrow{6}$&$\{te_1,\,e_2,\,tf_1,\,f_2,\,f_3\}$&
${43}^{\gamma}\rightarrow{32}$ & $\{e_{1},\,e_{2},\,t^{-1}f_{1},\,f_{3},\,f_{2}\},\, \gamma(t)=t^{-1}$ \\\hline
${26}\rightarrow{23}$& $\{e_1,\,e_2,\,tf_2,\,t^{-1}f_3,\,f_1\}$ &${43}^{\gamma}\rightarrow{39}$& $\{e_{1},\,e_{2},\,f_{1},\,f_{2},\,t^{-1}f_{3}\},\, \gamma(t)=t$\\\hline

 ${26}\rightarrow{25}$& $\{e_1,\,e_2,\,f_1,\,f_2,\,tf_3\}$&${43}^{\gamma}\rightarrow{42}$ & $\{e_{1},\,e_{2},\,f_{1},\,f_{2},\,f_{3}\},\,\gamma(t)=t$\\\hline
 
    ${27}\rightarrow{5}$&$\{te_1,\,te_2,\,tf_1,\,f_2,\,f_3\}$&${44}^{\phi}\rightarrow {20}$& $\{-te_{1},\,te_{2},\,tf_{1},\,-tf_{3},\,f_{2}\},\, \phi(t)=t^{-1}$ \\\hline  ${27}\rightarrow{26}$&$\{e_1,\,e_2,\,tf_1,\,tf_2,\,t^{-1}f_1+f_3\}$ &${44}^{\phi}\rightarrow {34}$& $\{-e_{1},\,e_{2},\,t^{-1}f_{1},\,f_{3},\,f_{2}\},\, \phi(t)=t^{-1}$\\\hline

     ${28}\rightarrow{7}$&$\{te_1,\,te_2,\,tf_1,\,f_2,\,f_3\}$&${44}^{\phi}\rightarrow {40}$ & $\{te_{1},\,t^{2}e_{2},\,t^{3}f_{1},\,tf_{2},\,tf_{3}\},\, \phi(t)=t$ \\\hline
      ${28}\rightarrow{27}$& $\{e_1,\,e_2,\,f_1,\,f_2+f_3,\,tf_3\}$&${44}^{\phi}\rightarrow {43}^{\gamma}$ & $\{e_{1},\,e_{2},\,tf_{1},\,tf_{2},\,tf_{3}\}$, $\phi(t)=\gamma_{0}$, $\gamma_{0}\in\mathbb{C}^{\ast} \ \backslash \ \{1\}$ \\ \hline
    

 \end{longtable} 
}
{\footnotesize
\begin{longtable}{|l|l|}
 \caption{Non-degenerations between nilpotent Jordan superalgebras of type $(2,3)$.}\label{table:ndeg23}\\
\hline  $J\nrightarrow J^{\prime} $  & Reason\\ \hline 

$3,21\nrightarrow4;\,\,\,\, 5,8,17,26\nrightarrow11,13;\,\,\,\, 7,9,36,37\nrightarrow11,13,16;$&\multirow{11}{*}{$F(J)\nrightarrow F(J^{\prime})$} \\ 
$10\nrightarrow 11-14,16,17,19;\,\,\,\, 14\nrightarrow11,16;$& \\
$18\nrightarrow 11-14,19;$&\\
$24\nrightarrow 4,5,6,8,11,13,16,25,26,35;$& \\
$27,29\nrightarrow11,13,16,35;$&\\
$28,30\nrightarrow11-14,16,17,19,32,35,36,37,39,42;$&\\
$33\nrightarrow11,12,15,19,20,35-40,42;$&  \\  
$34\nrightarrow 11-14,19,39,42;$&\\
$38\nrightarrow11-14,16,17,19,32,39,42;$ &  \\ 
$40\nrightarrow13,16,17,19,32,42;$&\\
$41\nrightarrow13-20,32,34,42.$&\\  

  \hline
$3,6,11,13\nrightarrow21,23;\,\,\,\, 4\nrightarrow21;$&\multirow{4}{*}{$J_{0}\nrightarrow J_{0}^{\prime} $}\\
$5,8,16\nrightarrow21,22,23,25;$&\\
$7,9,12,14,17,19\nrightarrow 21,22,23,25,26,35;$&\\
$10,15,18,20 \nrightarrow 21-27,29,32,35,36,37,39,42.$&\\
\hline
$4,11,13\nrightarrow2;$&\multirow{11}{*}{$a(J)\nrightarrow a(J^{\prime})$}\\
$5,7,8,12,14,17,26,27,29,36,37\nrightarrow3;$&\\
$15\nrightarrow2,3,5-9,12,14,17;$& \\ 
$16\nrightarrow2,3,6;\,\,\,\, 18,20\nrightarrow 3,9;$&\\
$19\nrightarrow 2,3,5,6,8;\,\,\,\, 25\nrightarrow2,23;$&\\
$28,34,38,40\nrightarrow 3,9,24;$&\\
$32,39,42\nrightarrow2,3,5,6,8,22,23,26;$&\\
$33,41\nrightarrow3,9,10,24,30;$& \\  
$35\nrightarrow2,3,6,22,23.$&\\
\hline 
$7,17,18,27,28,33,34,37\nrightarrow8.$&\multirow{1}{*}{$p(y_1,y_2,y_3)=(y_1,y_2)y_3$}\\ 
\hline
$12,29,36,37\nrightarrow 5;\,\,\,\, 38,40\nrightarrow5,7,27;$&\multirow{2}{*}{$p(x_1,y_1,y_2)=(x_{1}y_{1})y_{2}$}\\
$41\nrightarrow5,7,27,28.$&\\
\hline
$12,39\nrightarrow13,16;\,\,\,\, 31^{\lambda}\nrightarrow14-20,32,33,34,42,{43}^{\gamma};$&\multirow{3}{*}{$\mathrm{dim}(J^{2})_{1}< \mathrm{dim}({J^{\prime}}^{2})_{1}$}\\
$35\nrightarrow13;\,\,\,\, 41\nrightarrow{43}^{\gamma},44^{\phi}.$&\\
\hline
$16\nrightarrow 11,13;\,\,\,\, {31}^{\lambda}\nrightarrow11,12,13,35-41,44^{\phi};$ & \multirow{2}{*}{$\mathrm{dim}(\mathrm{Ann}(J))_{0}> \mathrm{dim}(\mathrm{Ann}(J^{\prime}))_{0}$}\\
$32\nrightarrow11,13,35;\,\,\,\, 33\nrightarrow13,14,{43}^{\gamma},44^{\phi};\,\,\,\,34\nrightarrow35.$&\\
\hline
$20,34\nrightarrow7;\,\,\,\, 35\nrightarrow11;\,\,\,\, 41\nrightarrow 31^{\lambda};\,\,\,\,42\nrightarrow11,16;$& \multirow{2}{*}{$\mathrm{dim}(\mathrm{Ann}(J))_{1}> \mathrm{dim}(\mathrm{Ann}(J^{\prime}))_{1}$}\\ 
${44}^{\phi}\nrightarrow 3,7,9,10,18,24,28,30,31^{\lambda},33.$& \\
\hline
$26,27\nrightarrow22;\,\,\,\, 28,33\nrightarrow22,29;\,\,\,\, 34\nrightarrow36;$&\multirow{2}{*}{$ \mathrm{nilindex}(J) < \mathrm{nilindex}(J^{\prime}) $}\\
${43}^{\gamma}\nrightarrow9,14,22,24,29,37; \,\,\,\, {44}^{\phi}\nrightarrow 38,41.$&\\
\hline

$30\nrightarrow7.$ & $\mathcal{R} = \left\{ 
\begin{array}{l}
\rho_{1j}^{k}=\rho_{2j}^{k}=\Gamma_{23}^{l}=0,\\
j\in\{2,3\}, \ k\in\{1,2,3\} \ \text{and} \ l\in \{1,2\}
\end{array}
\right\}$\\ \hline

$33\nrightarrow29,31^{\lambda};\,\,\,\, 34,40,44^{\phi}\nrightarrow22,29,37;\,\,\,\, 36\nrightarrow22.$&$\mathrm{dim}(J^{2})_{0}< \mathrm{dim}({J^{\prime}}^{2})_{0}$ \\ \hline
   
${43}^{\gamma}\nrightarrow 2,3,5-8,12,16,17,23,27,36.$& $J$ is a Jordan algebra, while $J^{\prime}$ is not. \\ \hline
${44}^{\phi}\nrightarrow15.$&$\mathrm{dim}(\mathrm{Z}(J))_{0}\leq \mathrm{dim}(\mathrm{Z}(J^{\prime}))_{0}$\\ \hline

\end{longtable}
    
}

\begin{figure}[H]
\centering
\includegraphics[scale=0.2]{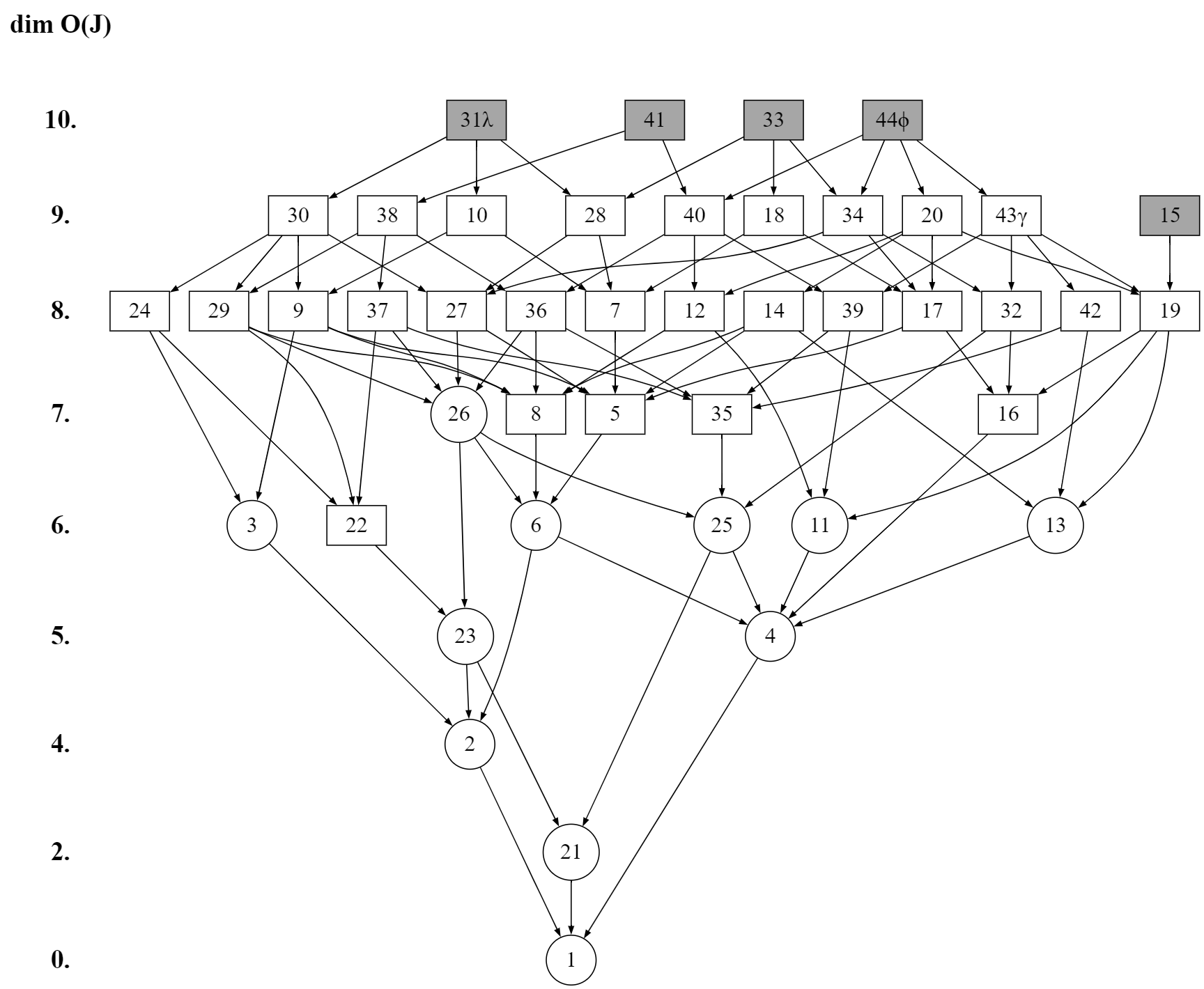}
\caption{Hasse diagram of degenerations for nilpotent Jordan superalgebras of type $(2,3)$.}
\label{fig:Hasse23}
\end{figure}

\end{document}